\documentclass[11pt]{article}

\usepackage[margin=1in]{geometry}

\usepackage{amsmath, amssymb, amsthm, bm}
\usepackage{mathtools, bbm}
\usepackage{comment, booktabs, multirow}

\usepackage[bookmarks, colorlinks=true, plainpages = false,
            citecolor=red,
            linkcolor=blue,
            anchorcolor=red,
            urlcolor=blue]{hyperref}

\usepackage{parskip}
\usepackage{subfigure, float}

\newtheorem{theorem}{Theorem}
\newtheorem{corollary}{Corollary}
\newtheorem{lemma}{Lemma}
\newtheorem{proposition}{Proposition}

\newtheorem{remark}{Remark}
\newtheorem{definition}{Definition}

\newcommand{\E}{\mathbb{E}}

\renewcommand{\d}{\mathrm{d}}
\renewcommand{\P}{\mathbb{P}}
\newcommand{\Q}{\mathbb{Q}}
\renewcommand{\i}{\mathrm{i}}

\renewcommand{\hat}{\widehat}

\newcommand{\Var}{\mathrm{Var}}
\newcommand{\TV}{\mathrm{TV}}
\newcommand{\KL}{\mathrm{D_{KL}}}

\newcommand{\uniform}{\mathrm{uniform}}
\newcommand{\Binom}{\mathrm{binom}}
\newcommand{\const}{\mathrm{const}}
\newcommand{\fa}{\mathrm{fa}}

\newcommand{\miss}{\mathrm{miss}}
\newcommand{\ones}[1]{\mathbf{1}\left\{#1\right\}}
\newcommand{\stepa}[1]{\stackrel{\mathrm{(a)}}{#1}}
\newcommand{\stepb}[1]{\stackrel{\mathrm{(b)}}{#1}}

\title{Detecting Planted Structure in Circular Data}
\author{
    Taha Ameen and Bruce Hajek 
    \thanks{
        The authors are with the Department of Electrical and Computer Engineering and the Coordinated Science Laboratory, University of Illinois, Urbana, IL 61801, USA. E-mails: \{ \texttt{tahaa3, b-hajek} \} \texttt{@illinois.edu} 
    }
}
\date{}

\begin{document}
\maketitle

\begin{abstract}
    Hypothesis testing problems for circular data are formulated, where observations take values on the unit circle and may contain a hidden, phase-coherent structure. Under the null, the data are independent uniform on the unit circle; under the alternative, either (i) a planted subset of size $K$ concentrates around an unknown phase (the flat setting), or (ii) a planted community of size $k$ induces coherence among the edges of a complete graph (the community setting). In each of the two settings, two circular signal distributions are considered: a hard-cluster distribution, where correlated planted observations lie in an arc of known length and unknown location, and a von Mises distribution, where correlated planted observations follow a von Mises distribution with a common unknown location parameter. For each of the four resulting models, nearly matching necessary and sufficient conditions are derived (up to constants and occasional logarithmic factors) for detectability, thereby establishing information-theoretic phase transitions. 
\end{abstract}


\tableofcontents

\section{Introduction} \label{sec-introduction}

Many modern data sets are inherently directional, taking values on the unit circle rather than the real line. Examples include phases of oscillatory signals, times of day or year (viewed modulo 24 hours or one year), headings of moving objects, and phase differences between pairs of time series. In such settings, a natural statistical question is whether observed angles are unstructured (compatible with a uniform distribution on the circle), or whether there is a hidden subset that exhibits a preferred direction or phase.

This paper studies circular  data in the presence of planted structure. Under the null hypothesis, all observations are i.i.d. uniform on the circle, while under the alternative, there is a small unknown subset of indices (or community of vertices in a network) whose angles concentrate around an unknown phase. These hypothesis testing problems combine two classical themes: (i) circular statistics, where the von Mises distribution serves as a ``Gaussian on the circle''~\cite{mardia2009directional}, and (ii) planted structure and phase transitions in high-dimensional detection, such as planted clique and community detection in stochastic block models.

Our goal is to characterize when detection is information-theoretically possible in directional planted models. We describe three application domains where such problems arise.
\begin{enumerate}
    \item \textit{Seasonal and diurnal effects in event data.} 
    Consider time-stamped events (transactions, failures, attacks, messages) where $X_i\in[0,2\pi)$ denotes the time of the $i$th event (modulo 24 hours or one year). Many applications exhibit background activity spread uniformly across all times, alongside a smaller subset of events driven by a seasonal or diurnal mechanism (failures that concentrate during hot summer days, or cyber-attacks that concentrate at night). Under a natural null model, all event times are i.i.d.\ uniform on $[0,2\pi)$; under the alternative, an unknown subpopulation of size $K$ has times that concentrate in an unknown arc of length $2\pi\tau$ (for $\tau < 1$), or around an unknown preferred phase with concentration $\kappa$. Here $\tau$ or $\kappa$ quantifies the strength of concentration, and $K$ the size of the affected population. 
    \item \textit{Phase-locked oscillators.} 
    Suppose we measure the instantaneous phases $X_1,\dots,X_N\in[0,2\pi)$ of $N$  oscillators (neurons in a neural population, devices in a power grid, or Kuramoto oscillators modeling synchronization phenomena). When oscillators are fully desynchronized, their phases are close to uniform on the circle. However, when a subset of $K$ oscillators is driven by a common source, their phases become localized near a common direction. The strength of this localization can be modeled by a concentration parameter $\kappa$, with the von Mises distribution providing a canonical alternative to uniformity. 
    \item \textit{Phase-coherent communities in networks.} 
    Consider a network of $n$ time series (EEG channels, financial assets, or sensor readings). At a given frequency, the pairwise relationship between series $i$ and $j$ can be summarized by their cross-spectral coherence, a complex-valued quantity whose phase $X_{ij}\in[0,2\pi)$ encodes the phase offset between the two signals. This yields a matrix $(X_{ij})_{1\leq i<j\leq n}$ of angular measurements on the edges of the complete graph. Under unstructured interactions, these edge phases are approximately uniform. However, if a community $C^*$ of $k$ nodes is driven by a common oscillatory source (e.g. a shared rhythm in sensor networks), then the phases $X_{ij}$ for $i,j\in C^*$ become aligned around some offset $\Theta^*$ and exhibit significantly greater coherence than background edges. 
\end{enumerate}
The task is to determine when a phase-coherent subset (or community) can be detected from the observed circular data, as a function of the population size ($N$ or $n$), the subset size ($K$ or $k$), and the concentration strength ($\tau$ or $\kappa$).

\subsection{Models, tests and objectives}
A random variable $X$ is said to have the von~Mises distribution with location parameter $\Theta^*\in[0,2\pi)$
and concentration parameter $\kappa>0$, denoted $X\sim \mathrm{vonMises}(\Theta^*,\kappa)$, if it has density
\[
    f(\theta;\Theta^*,\kappa)
    = \frac{1}{2\pi I_0(\kappa)}\exp\!\big(\kappa \cos(\theta-\Theta^*)\big) \, ,
    ~~~~~~ \theta\in[0,2\pi) \, ,
\]
where $I_0$ is the modified Bessel function of the first kind of order $0$,
\(
    I_0(\kappa)=\frac{1}{2\pi}\int_{0}^{2\pi}\exp\!\big(\kappa\cos u\big)\, \d u \, .
\)

We introduce four hypothesis testing problems below, with model parameters listed after the model names. In all models, the null hypothesis $H_0$ assumes directional homogeneity (i.i.d.\ uniform data on the circle), while the alternative $H_1$ introduces a structured component.

\begin{enumerate}
    \item \textbf{Hard-cluster flat model ($N,K,\tau$).} Angles $X_1,\dots,X_N\in[0,2\pi)$ are observed. Under $H_0$, the $X_i$ are independent and uniform on $[0,2\pi)$. Under $H_1$, there is an unknown planted subset $S^*\subset[N]$ of size $K$ and an unknown arc location $\Theta^*\in[0,2\pi)$ such that
    \[
        X_i \sim \mathrm{Uniform}([\Theta^*,\Theta^* + 2\pi\tau]) ~~~~ \mbox{for } i\in S^*,
        \qquad
        X_i \sim \mathrm{Uniform}([0,2\pi]) ~~~~ \mbox{for } i\notin S^* \, ,
    \]
    independently. This captures a hard cluster of events or oscillators localized in an unknown window of width $2\pi\tau$ on the circle.

    \item \textbf{Von Mises flat model ($N,K,\kappa$).} Angles $X_1,\dots,X_N\in[0,2\pi)$ are observed, with the same null $H_0$ as in the hard-cluster flat model. Under $H_1$, there is a planted subset $S^*$ of size $K$ and a preferred phase $\Theta^*$ such that
    \[
        X_i \sim \mathrm{vonMises}(\Theta^*,\kappa) ~~~~ \mbox{for } i\in S^* \, ,
        \qquad
        X_i \sim \mathrm{Uniform}([0,2\pi]) ~~~~ \mbox{for } i\notin S^*,
    \]
    independently. 

    \item \textbf{Hard-cluster community model ($n,k,\tau$).} A symmetric matrix of angles $(X_{ij})_{1\leq i, j\leq n}$ is observed, viewed as angular labels on the edges of the complete graph on $n$ vertices. Under $H_0$, the $X_{ij}$ for $i<j$ are independent uniform on $[0,2\pi)$. Under $H_1$, there is an unknown community $C^*\subset[n]$ of size $k$ and an unknown phase $\Theta^*$ such that
    \[
        X_{e} \sim \mathrm{Uniform}([\Theta^*,\Theta^* + 2\pi\tau]) ~~~~ \mbox{for } e\in E(C^*) \, ,
    \]
    where $E(C) \triangleq \{e = \{i,j\}: i,j\in C,~i<j\}$ is the set of all edges with both end-points in $C$, and the angles for all edges with at least one endpoint outside $C^*$ remain uniform on $[0,2\pi)$. 

    \item \textbf{Von Mises community model ($n,k,\kappa$).} A symmetric matrix of angles $(X_{ij})_{1\leq i,j\leq n}$ is observed. Under $H_0$ the $X_{ij}$ for $i<j$ are independent uniform on $[0,2\pi)$. Under $H_1$, there is an unknown community $C^*\subset [n]$ of size $k$ and an unknown phase $\Theta^*$ such that
    \[
        X_{e} \sim \mathrm{vonMises}(\Theta^*,\kappa) ~~~~ \mbox{for } e\in E(C^*) \, ,
    \]
    whereas all other edges remain uniform on $[0,2\pi)$.
\end{enumerate}

In each model, under $H_1$ it is assumed that the planted set $S^*$ (respectively $C^*$) is drawn uniformly at random among all subsets of size $K$ (respectively $k$), and that the planted phase $\Theta^*$ is drawn independently from the uniform distribution on $[0,2\pi)$.  The law of the observations under $H_1$, denoted by $\P$, is then the resulting mixture over $(S^*,\Theta^*)$ (respectively $(C^*,\Theta^*$)), and $\Q$ denotes the corresponding law under $H_0$.

The four models above are studied in parallel because they represent related manifestations of directional structure.  The hard-cluster and von~Mises models correspond to two canonical ways of modeling concentration on the circle, and in the regime of small window widths $\tau$ or large concentration parameters $\kappa$ they are closely related, with the rough correspondence $\tau \approx \kappa^{-1/2}$.  Likewise, the flat and community models differ primarily in their combinatorial structure, while sharing many of the same probabilistic mechanisms and bounding techniques.  We study these models side by side to highlight which phenomena are specific to a particular geometry and which are intrinsic to directional planted structure more broadly.

We now formalize the detection problem. A \textit{test} is a measurable function $\phi$ mapping data to $\{0,1\}$, where $\phi(X)=1$ means rejecting $H_0$ in favor of $H_1$.  The \emph{false alarm} and \emph{miss} probabilities are defined by
\[
    p_{\fa}(\phi) \,\triangleq\, \Q\{\phi(X)=1\} \, ,
    \qquad
    p_{\miss}(\phi) \,\triangleq\, \P\{\phi(X)=0\} \, .
\]
We will be interested in the following detection regimes as $n$ (or $N$) tends to infinity.

\medskip 

\begin{definition}[Strong and weak detection]
\label{def:strong_weak_detection}
\hfill
\begin{itemize}
    \item \emph{Strong detection is possible} if there exists a sequence of tests $(\phi_n)$ such that $p_{\fa}(\phi_n) \to 0$ and $p_{\miss}(\phi_n) \to 0$.

    \item \emph{Weak detection is possible} if there exists a sequence of tests $(\phi_n)$ and a constant $\delta>0$ such that $p_{\fa}(\phi_n) + p_{\miss}(\phi_n) \leq 1 - \delta$ for all sufficiently large $n$.

    \item \emph{Weak detection is impossible} if for every sequence of tests $(\phi_n)$, $p_{\fa}(\phi_n) + p_{\miss}(\phi_n) \geq 1 - o(1)$. Equivalently, weak detection is impossible if the total variation distance between $\P$ and $\Q$ tends to zero.
\end{itemize}
\end{definition}

\paragraph{Interval and coherence tests}
We propose and analyze natural testing procedures tailored to directional planted structure. For the flat models, we study an interval test with window length $2\pi\tau$ and threshold $\gamma$. Given observations $X_1,\dots,X_N\in[0,2\pi)$, define the statistic 
\begin{align} \label{eq: interval-statistic}
    T_{\mathrm{int}}(X)
    \triangleq
    \sup_{\theta \in [0,2\pi)}
    \sum_{i=1}^N \mathbf{1}\{X_i \in [\theta,\theta + 2\pi\tau]\} .
\end{align}
The interval test for the flat models with parameters $(\tau, \gamma)$  rejects $H_0$ if $T_{\mathrm{int}}(X) \geq \gamma$.

Similarly, for the community models, we study an interval test with window length $2\pi\tau$. Given observations $(X_{ij})_{1 \leq i < j \leq n}$, consider the statistic
\begin{align} 
    T_{\mathrm{int}}^{\mathrm{comm}}(X) \triangleq  \sup_{\theta \in [0,2\pi)}  \,\, \sup_{C\subseteq [n] : |C| = k} 
    \mathbf{1}\left\{ X_e \in [\theta, \theta + 2\pi\tau] \, \mbox{ for all}~ e \in E(C)\right\} \,.
\end{align}
The community interval test with parameters $(k,\tau)$ rejects $H_0$ if $T_{\mathrm{int}}^{\mathrm{comm}}(X) =1$.

For the von Mises community model, we also study a coherence test. For each edge $e = \{i,j\}$, let $Z_e = e^{\mathrm{i} X_{e}}$ (with $\mathrm{i} \triangleq \sqrt{-1}$), and define the coherence statistic
\begin{align}
    T_{\mathrm{coh}}(X)
    \triangleq 
    \max_{C\subset[n]:\,|C|=k} \, \bigg|\sum_{e\in E(C)} Z_{e}\bigg| \, ,
\end{align}
The coherence test with parameters $(k,\beta)$ rejects $H_0$ if $T_{\mathrm{coh}}(X) \geq \beta$.

The interval and coherence tests correspond to  generalized likelihood ratio tests for the hard-cluster and von Mises family of detection problems, respectively. A derivation is presented in Appendix~\ref{app: GLRT}. 
The coherence test is related to the Rayleigh test in circular statistics, which compares the sum of $Z_e$ over \emph{all} edges $e$ to a threshold. This statistic, although efficiently computable, performs worse than the coherence test at detecting planted communities; see Appendix~\ref{app: poly-time}. 
Finally, the interval and coherence tests are themselves related; see Appendix~\ref{app-relationship-test}.

\subsection{Related work}

\paragraph{Directional statistics}
Directional statistics concerns data on circles or spheres (angles, phases), and has well-developed tools for modeling and inference~\cite{mardia2009directional, fisher1995statistical}. A central problem in this area is testing circular uniformity: given i.i.d.\ angles $\theta_1,\dots,\theta_n\in[0,2\pi)$, decide whether they are consistent with $\mathrm{uniform}([0,2\pi))$ or exhibit some systematic departure. Classical tests include the classical Rayleigh test~\cite{rayleigh1880resultant} (related to our coherence test), as well as Kuiper’s test~\cite{kuiper1960tests} and Watson’s test~\cite{watson1961goodness} (comparing empirical and uniform CDFs), and Rao’s spacing test~\cite{rao1976tests} (based on gaps between ordered angles). These methods (and their Bayesian counterparts~\cite{landler2018circular,mulder2021bayesian}) test goodness-of-fit in the homogeneous setting where all observations are drawn from a single common distribution, and the goal is to detect any global non-uniformity. A related contamination model is studied in~\cite{bentley2006modelling}, which fits a mixture of von Mises and uniformly distributed data, and uses likelihood ratio and Watson-type goodness-of-fit procedures to choose between uniform, von Mises, and mixture fits. Here, each observation is drawn i.i.d.\ from the same mixture.

In contrast, our setting differs in two ways.  First, under the alternative only a small, unknown subset of observations (or edges, in the community models) is structured, while the majority remain uniform. Second, we focus on high-dimensional regimes in which the problem size and planted subset size grow, and we ask for information-theoretic thresholds that separate detectability from impossibility. Classical uniformity tests typically have little power in such sparse planted regimes, where the global empirical distribution can remain near-uniform despite the coherent subset.

Beyond hypothesis testing, scan statistics have been studied in other contexts as well. In the circular setting, the scan statistic is the maximum number of observations contained in any arc of a given length as the arc sweeps around the circle.  Foundational work of Cressie~\cite{cressie1977scan,cressie1980asymptotic} analyzes the distribution of this statistic under circular uniformity, and broader developments and extensions are reviewed in the monograph~\cite{glaz2005scan}.  Scan statistics are also used in applications with periodic structure (e.g., seasonality), where one searches for a localized pulse of events by scanning a window around the year~\cite{wallenstein1989pulse, wallenstein1993power}. 

\paragraph{Planted structures in non-circular data} From a high-dimensional testing viewpoint, our models belong to the broad class of problems with planted structure, where a low-complexity signal (a sparse subset, a community, or a low-rank perturbation) is embedded in a high-dimensional noise background. Canonical examples of such problems include planted clique~\cite{jerrum1992large, alon1998finding, feige2000finding}, sparse PCA~\cite{dAspremont2004direct, johnstone2009consistency, deshpande2014information}, and sparse submatrix detection~\cite{butucea2013submatrix}; see Wu and Xu~\cite{wu2021planted} for a recent survey of information-theoretic and computational thresholds in such models. Particularly relevant in the graph setting are detectability and recovery thresholds for community detection in sparse stochastic block models, e.g.~\cite{mossel2018proof, massoulie2014community, banks2016information}, and the survey~\cite{abbe2018community}. As a special case, the \textit{planted dense subgraph} problem studies the detection of an anomalously dense vertex subset embedded in an otherwise homogeneous random graph, and has been analyzed from both statistical and computational perspectives~\cite{arias2014community, verzelen2015community, deshpande2015finding, hajek2017information}. Our work can be viewed as a directional analog of these models: we replace scalar observations (Bernoulli or Gaussian) with circular valued measurements, and characterize the corresponding detection and impossibility regimes for planted clusters and communities. Our community models also differ from stochastic
block models for a planted dense community in that for our models the
marginal distributions of the edge labels are the same under
either hypothesis.

A related line of work studies statistical--computational gaps in planted sparse structure problems, where detection is information-theoretically possible but conjecturally hard for polynomial time algorithms. For instance, Berthet and Rigollet~\cite{berthet2013complexity} gives evidence for such gaps for sparse PCA via average case reductions from planted clique, and Brennan, Bresler, and Huleihel~\cite{brennan2018planted} broadens this reduction-based program to obtain tight conditional lower bounds for several canonical planted models. Our focus here is on information-theoretic detectability and impossibility in directional planted models; extending these reduction frameworks to circular-valued observations is a natural direction for future work.

\subsection{Summary and discussion of main results}

\paragraph{Information-theoretic detection thresholds} 
For each of the four models, we derive sharp conditions (up to constants and occasional logarithmic factors) on the triplets $(N,K,\tau)$, $(N,K,\kappa)$, $(n,k,\tau)$, and $(n,k,\kappa)$ that separate detectability from impossibility. Table~\ref{tab:recovery_conditions} summarizes these results for specific regimes, with more general conditions for achievability and impossibility listed in the theorem statements. (Functions $A(\kappa)$ and $R(\kappa)$ are defined in Appendix~\ref{app:function_bounds}.) 
Our results assume that the decision maker does not know the planted phase $\Theta^*$; see Appendix~\ref{app: knowing-theta} for a discussion of how access to $\Theta^*$ can improve performance in the flat case, and Appendix~\ref{app: SBM} for a discussion on the connection to stochastic block models in the community case.    

\begin{table}[t]
    \centering
    \renewcommand{\arraystretch}{1.5}
    \setlength{\tabcolsep}{9pt}
    \caption{Summary of necessary and sufficient conditions for detection}
    \label{tab:recovery_conditions}
    \vspace{1 em}
    \begin{tabular}{@{}llll@{}}
        \toprule
        \textbf{Model} & \textbf{Regime} & \textbf{Impossibility} & \textbf{Achievability} \\ \midrule
        
        \multirow{3}{*}{\shortstack[l]{\textbf{Hard-cluster}\\\textbf{flat}\\($N, K, \tau$)}} 
        & $K$ constant & $\tau = \omega\big(N^{-1 - \frac{1}{K-1}}\big)$ & $\tau = o\big(N^{-1 - \frac{1}{K-1}}\big)$ \\
        & $K = N^\alpha, \alpha \le 1/2$ & $\tau = \frac{K^2}{c N \log N}, \ c < 1-2\alpha$ & $\tau = \frac{K^2}{c N \log N}, \ c > 2$ \\ 
        & $K = N^\alpha, \alpha > 1/2$ & --- &  $\tau \in (0,1-\epsilon]$, $\epsilon > 0$ \\ \midrule
        
        \multirow{2}{*}{\shortstack[l]{\textbf{Von Mises}\\\textbf{flat}\\($N, K, \kappa$)}} 
        & $K = N^\alpha, \alpha \le 1/2$ & $\frac{1}{\sqrt{\kappa}} \ge \frac{K^2}{c N \log N}$ ($c<\frac{1-2\alpha}{2\sqrt{\pi}}$) & $\frac{1}{\sqrt{\kappa}} \le \frac{K^2}{c N \log N}$ ($c>0.5057$) \\
        & $K = N^\alpha, \alpha > 1/2$ & --- & $\kappa \geq\epsilon$, $\epsilon > 0$  \\ \midrule
        
        \multirow{4}{*}{\shortstack[l]{\textbf{Hard-cluster}\\\textbf{community}\\($n, k, \tau$)}} 
        & $3\leq k \leq n$ & --- & $\tau \leq \left( \frac k {ne} \right)^{\frac{(1+\epsilon)k}{\binom k 2 -1}}$\\
        & $k = c\log n$ & $\tau > e^{-2/c}$ &  $\tau \leq  e^{-2/c}$ \\
        & $k = o(\sqrt{n})$ & $\tau \ge \exp\left(-\frac{(2-\epsilon)\log(n/k^2)}{k-1}\right)$ & $\tau \leq \left( \frac k {ne} \right)^{\frac{(1+\epsilon)k}{\binom k 2 -1}}$\\
        & $k = \omega(\log{n})$ & $\tau \geq 1 - \frac{2(1-\epsilon)}{k-1}\log\left(1 + \frac{\epsilon_n n}{k^2}\right)$
        & $\tau \leq 1 - \frac{2(1+\epsilon)}{k-1}\log \frac{ne}{k}$  \\
        
        \midrule 
        
        \multirow{3}{*}{\shortstack[l]{\textbf{Von Mises}\\\textbf{community}\\($n, k, \kappa$)}} 
        & $3/\epsilon < k = o(\log n)$ & $\kappa \leq \left(\frac{n}{k^2}\right)^{\frac{4-\epsilon}{k-1}}$ & $\kappa \! \geq \! \left(\frac{4}{\pi^2} \!+\! \epsilon\right)  \!  \left(\frac{n}{k}\right)^{\frac{4(1+\epsilon)}{k-1}} \log k$ \\[0.5em]
        & $k=c\log n$     &  $c \log R(\kappa) < 2$ & \parbox[t]{0.24\linewidth}{\raggedright $c\, A^2(\kappa) > 2$ and $c > 2$, \\[0.3em] or $\kappa \geq \frac{2\log k}{1-\cos(\pi \exp(-2/c) )}$} \\[2.5em]
        & $k = \omega(\log n)$  & \parbox[t]{0.22\linewidth}{\raggedright $\kappa \leq \sqrt{\frac{4(1-2c-\epsilon)\log n}{k-1}}$, \\[0.2em] $k \leq n^c$ and $c < 1/2$} & $\kappa \geq (1+\epsilon)\sqrt{\frac{8\log n}{k-1}}$ \\
        \bottomrule
    \end{tabular}
\end{table}

A recurring theme in our results is a {small $K$ versus large $K$} dichotomy. For the {flat} models, there is a qualitative change around $K =  \sqrt{N}$: when $K=o(\sqrt{N})$, the planted subset is sufficiently sparse that detection requires the directions within the cluster to become increasingly concentrated (equivalently, $\kappa\to\infty$ in the von~Mises model, or $\tau\to 0$ in the hard-cluster model), whereas when $K=\omega(\sqrt{N})$ the number of planted observations makes detection possible even at {constant} concentration (e.g., $\kappa=\Theta(1)$ and a window width $\tau$ bounded away from $0$ and $1$).
In contrast, for the {community} models the corresponding transition occurs for much smaller correlated sets, around $k = \log n$: if $k=o(\log n)$ then strong detection forces the edge-level concentration to strengthen with $n$ (again $\kappa\to\infty$ or $\tau\to 0$), while if $k=\omega(\log n)$ the community contributes so many intra-community edges that detection can succeed with only bounded concentration.

In the small $k$ regimes, the hard-cluster and von Mises distributions are closely aligned: For large $\kappa$ the von Mises$(\Theta^*,\kappa)$ law is well approximated by the Gaussian distribution with mean $\Theta^*$ and standard deviation $\frac 1 {\sqrt{\kappa}}$ yielding the heuristic correspondence $ \tau \,\approx\, \kappa^{-1/2}$ as $\kappa\to \infty.$
This correspondence helps explain why the same scanning-based procedures and thresholds emerge in parallel across the hard-cluster and von~Mises models.  

The hard-cluster and von Mises distributions are also aligned in the  large $k$ regimes, where $\tau\to 1$ and $\kappa\to 0$, in the sense that both the uniform distribution on $[0,2\pi\tau]$ and the von Mises distribution with spread parameter $\kappa$ converge to the uniform distribution on $[0,2\pi].$  Comparing the achievability conditions for the two community models suggests the correspondence $1-\tau \approx {\kappa^2} /4.$
The KL divergences in this regime satisfy
\[
    \KL( \uniform ([0,2\pi\tau]) || \uniform ([0,2\pi]) = -\ln \tau = 1 - \tau + o(1-\tau)
\]
and 
\[
    \KL( \mathrm{von Mises}(0,\kappa) || \uniform ([0,2\pi]) = {\kappa^2}/{4} + o(\kappa^2),
\]
which suggests the same correspondence.

\paragraph{Methodology}
Our achievability results are obtained by analyzing the interval and coherence tests, combined with sharp tail bounds for the false alarm and miss probabilities. The converse results rely on second-moment methods and control of likelihood ratios, together with convex-order comparisons between hypergeometric and binomial overlap distributions.  In the von Mises models, the analysis uses asymptotics for modified Bessel functions and a Gaussian approximation for the von Mises distribution in the large-concentration regime. The techniques illustrate how directional geometry and planted combinatorial structure can be analyzed in a unified way.

\section{The hard-cluster flat model} \label{sec-hard-cluster-flat}

In this section and the next three sections, the main result is presented as a theorem, followed by corollaries specializing the necessary and sufficient conditions to various regimes of interest, and then followed by subsections giving the proof of the theorem.

\bigskip

\begin{theorem}
\label{thm:hard-cluster-flat}
Consider the hard-cluster flat model with parameters $(N,K,\tau)$.
\begin{itemize}
\item[$\mathrm{(a)}$] (Achievability)
Strong detection is achievable by the interval test with threshold $\gamma$ in either of the following two cases:
\begin{enumerate}
    \item[$\mathrm{(A1)}$] 
    $\gamma=K$ and
    \(
        \frac{N^K \, \tau^{K-1}}{(K-1)!} \to 0\, .
    \)

    \item[$\mathrm{(A2)}$] 
    There exist sequences $c_N\to\infty$ and $\gamma \equiv \gamma_N$ with
    \[
        \gamma_N \triangleq (N-K)\tau + K - c_N \sqrt{(N-K)\tau} \, ,
        \qquad
        \gamma_N \geq 1+(N-1)\tau \ \text{for all sufficiently large }N,
    \]
    such that
    \begin{align}
        & \frac{(K-1)^2(1-\tau)^2}{(N-K)\tau} \to \infty \, ,
        \label{eq:suff_cond_1}\\[0.3em]
        & \frac{\big(K(1-\tau) - c_N \sqrt{(N-K)\tau}\big)^2}
        {2N\tau + K(1-\tau) - c_N \sqrt{(N-K)\tau}} \,-\, \log N \, \to \,  \infty \, .
        \label{eq:suff_cond_2}
    \end{align}
\end{enumerate}
\item[$\mathrm{(b)}$] (Converse)
If $\tau \leq 1/2$ for sufficiently large $N$ and
\begin{align}
\label{eq:nec-condition-hard-cluster-flat}
    \frac{2N\tau^2}{K^2}\left(1+\frac{K}{N\tau}\right)^{K+1} = o(1) \, ,
\end{align}
then weak detection is impossible.
\end{itemize}
\end{theorem}

\bigskip


\bigskip 

\begin{corollary}[Achievability regimes]
\label{cor:flat_uniform_achievability}
Under the hard-cluster flat model, strong detection is possible in each of the following regimes with the interval test.
\begin{itemize}
    \item $K\geq 2$ fixed and $\tau=o\big(N^{-1-\frac{1}{K-1}}\big)$ 
    [Take $\gamma = K$ so (A1) is satisfied.]

    \item $K=N^{\alpha}$ with $0<\alpha\leq 1/2$, and $\ \tau=\frac{K^2}{(2+\epsilon)N\log N} = \frac{N^{2\alpha-1}}{(2+\epsilon)\log N}$
    for some constant $\epsilon\in(0,1/2)$. Here, the threshold $\gamma_N$ is chosen so that $c_N\to\infty$ and $c_N=o \big(\sqrt{\log N}\big)$.

    \item $K=N^{\alpha}$ with $1/2 < \alpha < 1$, and $\ \tau\in(0,1-\epsilon]$
    for some constant $\epsilon\in(0,1)$. Here, the threshold $\gamma_N$ is chosen so that $c_N \to \infty$ and $c_N=o\big(K/\sqrt{N\tau}\,\big)$.
\end{itemize}
\end{corollary}

\bigskip

\begin{corollary}[Impossibility regimes]
\label{cor:flat_uniform_impossibility}
Under the hard-cluster flat model, weak detection is impossible in each of the following regimes.
\begin{itemize}
    \item $K\ge 2$ fixed and $\tau=\omega\big(N^{-1-\frac{1}{K-1}}\big)$.
    [This scaling implies the condition \eqref{eq:nec-condition-hard-cluster-flat} of
    Theorem~\ref{thm:hard-cluster-flat}(b).]

    \item $K=N^{\alpha}$ with $0<\alpha< 1/2$, and
    \(
        \tau=\frac{K^2}{\epsilon\,N\log N}
        =
        \frac{N^{2\alpha-1}}{\epsilon\log N}
    \)
    for any constant $\epsilon\in(0,1-2\alpha)$.
    [For this choice of $(K,\tau)$,
    \[
        \frac{2N\tau^2}{K^2}
        = \frac{2}{\epsilon^2\log^2 N}\,N^{2\alpha-1} \, ,
        \qquad
        \frac{K(K+1)}{N\tau}
        = \epsilon\log N\,(1+N^{-\alpha}),
    \]
    hence
    \[
        \frac{2N\tau^2}{K^2} \, \left( 1 + \frac{K}{N\tau}\right)^{K+1} 
        \,\leq\,
        \frac{2N\tau^2}{K^2}\,\exp\left\{\frac{K(K+1)}{N\tau}\right\}
        \,=\,
        \frac{2+o(1)}{\epsilon^2\log^2 N}\,N^{2\alpha+\epsilon-1}
        \,\to \, 0 \, ,
    \]
    since $\epsilon<1-2\alpha$.]
\end{itemize}
\end{corollary}

\subsection{Achievability for hard-cluster flat model -- interval test}   \label{sec-2-achievability}

\noindent (a) 
Fix an integer threshold $\gamma \geq K$ and consider the interval test that thresholds the
statistic~\eqref{eq: interval-statistic}, i.e., it decides $H_1$ if there exists an interval of length $2\pi\tau$ containing at least $\gamma$ of the $X_i$'s, and decides $H_0$ otherwise.
This decision rule is equivalent to the generalized maximum likelihood decision rule obtained by comparing the maximum over $\theta$ of the conditional likelihood given $\Theta^*=\theta$ to a threshold; see Appendix~\ref{sec:scan_as_GLRT}.

\medskip
\noindent \emph{Case (A1): $\gamma=K$.}~~
Here the probability of missed detection is zero: under $H_1$ the interval $[\Theta^*,\Theta^*+2\pi\tau]$ contains all $K$ points indexed by $S^*$, so the scan statistic is at least $K$ and $p_{\miss}=0$. It remains to bound $p_{\fa}$ under $H_0$.

We upper bound the probability of false alarm $p_{\fa}$ as follows. Let $B_i$ be the event that the interval $[X_i,X_i+2\pi\tau]$ contains at least $\gamma$ points, including $X_i$ itself. Then $p_{\fa}=\Q(\cup_{i\in[N]}B_i)$, so by the union bound and the fact that $\Q(B_i)=\Q(B_1)$ for all $i$,
\[
    p_{\fa} \, \leq\,  N \cdot\Q(B_1) \, .
\]
Under $\Q$, the cardinality of $\{j\in\{2,\ldots,N\}:X_j\in[X_1,X_1+2\pi\tau]\}$ has the $\Binom(N-1,\tau)$ distribution, so the total number of observations in $[X_1,X_1+2\pi\tau]$ is distributed as $1+\Binom(N-1,\tau)$.
The event $B_1$ is the union of the events $\{X_j\in[X_1,X_1+2\pi\tau]\ \text{for } j\in S\}$ over all subsets $S\subset\{2,\ldots,N\}$ of cardinality $\gamma-1$. By another application of the union bound,
\[
    \Q(B_1) \, \leq \,\binom{N-1}{\gamma-1}\tau^{\gamma-1}.
\]
Hence,
\begin{align}   \label{eq:pfa_unif_up_bound1}
    p_{\fa}
    \,\leq\,
    N\binom{N-1}{\gamma-1}\tau^{\gamma-1}
    \,=\,
    \binom{N}{\gamma}\gamma\,\tau^{\gamma-1}
    \,\leq\,
    \frac{N^\gamma\,\tau^{\gamma-1}}{(\gamma-1)!} \, .
\end{align}
Specializing to $\gamma=K$, this shows that $p_{\fa}\to 0$ whenever
$ \frac{N^K \, \tau^{K-1}}{(K-1)!} \to 0\, .$ This concludes Case (A1).

\medskip
\noindent \emph{Case (A2): $\gamma=\gamma_N$.}~~
If $K$ increases with $N$ then the critical value of $\tau$ could be large enough that under $\P$ the interval $[\Theta^*,\Theta^*+2\pi\tau]$ contains not only the $K$ points indexed by $S^*$ but also, with high probability, a significant number of other points. This motivates the use of a threshold $\gamma$ that is greater than $K$. The probability of false alarm depends only on $N,\tau,$ and $\gamma$ and does not involve $K$.

We first bound $p_{\fa}$ using a Chernoff bound, which is suitable when $\gamma$ is at least the mean under $\Q$. As above, $p_{\fa}\le N\,\Q(B_1)$. Let $Y\triangleq 1+\Binom(N-1,\tau)$ denote the number of points falling in $[X_1,X_1+2\pi\tau]$ under $\Q$, with mean $\mu=1+(N-1)\tau$.
A standard Chernoff bound yields $P\{Y\geq (1+\delta)\mu\} \, \leq \, \exp\big(-\mu\delta^2/(2+\delta)\big)$, and therefore, for $\gamma \geq 1+(N-1)\tau$,
\begin{align*}
    \Q(B_1)
    \,= \,
    P\{1+\Binom(N-1,\tau)\geq \gamma\}
    \,\leq \,
    \exp\left\{
        -\frac{(\gamma-1-(N-1)\tau)^2}{1+(N-1)\tau+\gamma}
    \right\} \, . 
\end{align*}
Hence, for $\gamma\geq 1+(N-1)\tau$,
\begin{align} \label{eq:pfa_bnd}
    p_{\fa}
    \, \leq \,
    N \cdot \exp\left\{
        -\frac{(\gamma-1-(N-1)\tau)^2}{1+(N-1)\tau+\gamma}
    \right\} \, .
\end{align}
We next upper bound the probability of a miss. Suppose $H_1$ is true, so that the points $(X_i:i\in S^*)$ all lie in the interval $[\Theta^*,\Theta^*+2\pi\tau]$. Detection will occur if at least $\gamma-K$ of the points $(X_i:i\in[N]\setminus S^*)$ also lie in this interval. The number of such points has the $\Binom(N-K,\tau)$ distribution, so assuming $(N-K)\tau \geq \gamma-K$,
\begin{align} \label{eq:pmiss}
    p_{\miss}
    \, \leq \,
    P\{\Binom(N-K,\tau)<\gamma-K\}
    \, \leq \,
    \exp\left\{
        -\frac{((N-K)\tau-\gamma+K)^2}{2(N-K)\tau}
    \right\} \, ,
\end{align}
where we used the Chernoff bound
$P\{\Binom(\bar n,p) \leq (1-\delta)\bar n p\} \, \leq \, e^{-\bar n p\delta^2/2}$ with 
$\bar n=N-K$, $p=\tau$, and
$\delta=\frac{(N-K)\tau-\gamma+K}{(N-K)\tau} \, .$
Now suppose there exists a sequence $c_N\to\infty$ such that
\begin{align}
    &\frac{(K-1)^2(1-\tau)^2}{(N-K)\tau}\to\infty,
    \label{eq:suff_cond_1-repeat}\\
    &\frac{\bigl(K(1-\tau)-c_N\sqrt{(N-K)\tau}\bigr)^2}
    {2N\tau+K(1-\tau)-c_N\sqrt{(N-K)\tau}}-\log N\to\infty.
    \label{eq:suff_cond_2-repeat}
\end{align}
Select the threshold
\[
    \gamma=\gamma_N\triangleq (N-K)\tau+K-c_N\sqrt{(N-K)\tau},
\]
and assume (as in the theorem statement) that $\gamma_N\geq 1+(N-1)\tau$ for all sufficiently large $N$, so that \eqref{eq:pfa_bnd} applies. With this choice of $\gamma$, we have $(N-K)\tau-\gamma+K=c_N\sqrt{(N-K)\tau}$, and \eqref{eq:pmiss} implies
\[
    p_{\miss}\leq \exp\left(-{c_N^2}/{2}\right) \, \to \,  0 \, .
\]
Moreover, using $\gamma=\gamma_N$ in \eqref{eq:pfa_bnd} and noting that
$\gamma_N-1-(N-1)\tau = K(1-\tau)-c_N\sqrt{(N-K)\tau}-(1-\tau)$ 
and
$1+(N-1)\tau+\gamma_N = 2N\tau+K(1-\tau)-c_N\sqrt{(N-K)\tau}+O(1)$,
condition \eqref{eq:suff_cond_2-repeat} yields $p_{\fa}\to 0$.
This concludes Case (A2), and hence the proof of Theorem~\ref{thm:hard-cluster-flat}(a).

\medskip 

\subsection{Converse for hard-cluster flat model} \label{sec-2-converse}

We shall employ a second moment method for identifying conditions such that $\TV(\P,\Q)\to 0$, where $\TV(\P,\Q)$ is the total variation distance between $\P$ and $\Q$, thereby establishing the impossibility of weak detection.  The method takes advantage of the fact that $\P$ is a mixture of distributions which each have a  simple likelihood ratio with respect to $\Q.$  The method is described in~\cite{wu2021planted} and for the reader's convenience we explain it in detail here.  The key is to bound 
$\Var_{\Q}(L), $ where $L(X)=\frac{\d \P}{\d \Q}(X)$, which is also known as the $\chi^2$ divergence between $\P$ and $\Q$~\cite{tsybakov2009estimation}.
Since $\P\ll\Q$ it follows that
\[
    \TV(\P,\Q) = \frac 1 2 \, \E_{\Q} |L-1|  \leq \frac12 \sqrt{\Var_{\Q}(L)} \, .
\]
Hence, if $\Var_{\Q}(L)\to 0$ as $N\to\infty$, then $\TV(\P,\Q)\to 0$ and weak detection is impossible.   Since $E_{\Q}[L]=1$ it follows that $\Var_{\Q}(L) = \E_\Q[L^2] -1 .$

Averaging over the possible values of $(S^*,\Theta^*)$, we may write
\begin{align} \label{eq:L_as_mixture}
    L(X)
    =
    \frac{1}{2\pi}\int_0^{2\pi}
    \frac{1}{\binom{N}{K}}
    \sum_{\substack{S\subset[N] \, : \,  |S|=K}}
    L_{S,\theta}(X) \
    \d\theta \, ,
\end{align}
where
\[
    L_{S,\theta}(X)
    \triangleq
    \tau^{-K}\,
    \mathbf{1}\big\{X_i\in[\theta,\theta+2\pi\tau]\ \text{for all } i\in S\big\} 
\]
is the conditional likelihood ratio given $S^* = S$ and $\Theta^* = \theta$.
If we view $S$ and $\Theta$ as random variables independent of $X$ such that $S$ is uniformly distributed over subsets of $[N]$ with $|S|=K$, $\Theta$ is uniformly distributed over $[0,2\pi]$, and $S$ and $\Theta$ are mutually independent, then \eqref{eq:L_as_mixture} can compactly be rewritten as
$ L(X) =  \E_{S,\Theta}[L_{S,\Theta}(X)].$  Moreover, if $(S',\Theta')$ represents an independent copy of $(S,\Theta)$, then
\begin{align*}
L^2(X) =  \E_{S,\Theta}[L_{S,\Theta}(X)] \E_{S',\Theta'}[L_{S',\Theta'}(X)]
= \E_{S,\Theta, S', \Theta'}[L_{S,\Theta}(X)L_{S',\Theta'}(X)]
\end{align*}
Then, taking expectation with respect to $X$ having distribution $\Q$ and switching the order of integration yields:
\begin{align}  \label{eq:L2_as_mixture}
\E_{\Q}[L^2] = \E_{S,\Theta, S', \Theta'}\left[\E_Q[L_{S,\Theta}(X)L_{S',\Theta'}(X)]\right].
\end{align}
Equation \eqref{eq:L2_as_mixture} is the key equation used in this and other sections in this paper for calculating $\E_{\Q}[L^2]$ and $\Var_{\Q}(L) = \E_\Q[L^2] -1 .$

Fix $S,S',\theta,\theta'$, and let $j=|S\cap S'|$.
Let $2\pi\delta$ denote the length of the intersection of the two arcs
$[\theta,\theta+2\pi\tau]$ and $[\theta',\theta'+2\pi\tau]$ (modulo $2\pi$).
Under $\Q$, the coordinates $X_i$ are i.i.d.\ uniform on $[0,2\pi)$, and therefore
\begin{align*}
    \E_{\Q}\big[L_{S,\theta}(X) \cdot L_{S',\theta'}(X)\big]
    &=
    \tau^{-2K} \cdot 
    \Q 
    \left\{
    \begin{array}{rl}
        X_i \in [\theta,\theta+ 2\pi\tau] \cap [\theta',\theta'+2\pi\tau] & \mbox{ for } i\in S\cap S', 
        \\
        X_i \in [\theta,\theta+2\pi\tau] & \mbox{ for } i \in S \backslash S',
        \\
        X_i \in [\theta',\theta'+2\pi\tau] & \mbox{ for } i \in S' \backslash S
    \end{array} 
    \right\} 
    \\
    &=
    \tau^{-2K}\,\delta^{\,j}\,\tau^{2(K-j)}
    \,=\,
    \delta^{\,j}\,\tau^{-2j} \, .
\end{align*}

We next average over $\theta,\theta'$.
Let $u\in[0,1/2]$ denote the circular distance between $\theta$ and $\theta'$,
normalized by $2\pi$, i.e.
\[
    u \triangleq \frac{1}{2\pi}\min\big\{|\theta-\theta'|,
    \,
    2\pi-|\theta-\theta'|\big\},
    \qquad
    u\sim \mathrm{uniform}\left(\left[0, \, \frac12\right]\right) \, .
\]
Therefore, the normalized overlap $\delta$ is a function $\delta=\delta_\tau(u)$ given by
\begin{align} \label{eq:delta_general}
    \delta_\tau(u)
    =
    \begin{cases}
        (\tau-u)_+, & 0 < \tau \leq 1/2 \, ,\\[1mm]
        (2\tau-1) + (1-\tau-u)_+, & 1/2< \tau <1 \, ,
    \end{cases}
\end{align}
and thus, for each integer $j\geq 1$,
\begin{align} \label{eq:delta_moment}
    \E_{\theta,\theta'}[\delta^{\,j}]
    =
    2\int_0^{1/2}\delta_\tau(u)^{\,j}\,\d u
    =
    \begin{cases}
        \frac{2}{j+1}\,\tau^{j+1}, & 0<\tau\leq \frac12\\[2mm]
        \frac{2}{j+1}\Big(\tau^{j+1}-(2\tau \!-\! 1)^{j+1}\Big)
        +2\big(\tau - \frac12 \big) \big(2\tau \!-\! 1 \big)^j,
        & \frac12 < \tau <1
    \end{cases}
\end{align}
Now average over $S,S'$ as well.  Writing $J \triangleq |S\cap S'|$, we obtain
\begin{align} \label{eq:EQ_L2_general}
    \E_{\Q}[L^2]
    =
    \P\{J=0\}
    +
    \sum_{j=1}^K \P\{J=j\}\cdot 
    \tau^{-2j} \cdot 
    \E_{\theta,\theta'}\big[ \delta^{\,j} \big] \, , 
\end{align}
and $J$ has the hypergeometric distribution with parameters $(N,K,K)$.
Using $\Var_{\Q}(L)=\E_{\Q}[L^2]-1$ yields
\[
    \Var_{\Q}(L)
    =
    \sum_{j=1}^K \P\{J=j\}\,
    \Big(\tau^{-2j}\cdot \E_{\theta,\theta'}[\delta^{\,j}] - 1\Big)
    \,\leq \,
    \sum_{j=1}^K \P\{J=j\} \cdot 
    \tau^{-2j} \cdot \E_{\theta,\theta'}\big[\delta^{\,j} \big] \, .
\]
We specialize to the case $0< \tau \leq 1/2$ for the remainder of the proof.
Using \eqref{eq:delta_moment} in \eqref{eq:EQ_L2_general},
\begin{align*}
    \E_{\Q}[L^2]
    & \,=\,
    \P\{J=0\}
    \, + \,
    \sum_{j=1}^K \P\{J=j\}
    \left(\tau^{-2j}\cdot \frac{2}{j+1}\tau^{j+1}\right)\\
    & \,=\,
    \P\{J=0\}
    \, + \,
    \sum_{j=1}^K \P\{J=j\}\,
    \frac{2}{(j+1)\tau^{j-1}} \, ,
\end{align*}
so that
\begin{align} \label{eq:varQbnd}
    \Var_{\Q}(L)
    \, \leq \,
    \sum_{j=1}^K \P\{J=j\}\, \frac{2}{(j+1)\tau^{j-1}}.
\end{align}
We now compare the overlap distribution to a binomial.
The hypergeometric$(N,K,K)$ distribution is second-order stochastically less spread out than
$\Binom\left(K,\frac{K}{N}\right)$; see Hoeffding~\cite{hoeffding1963probability}.
Since $y\mapsto z^y$ is convex for each fixed $z\geq 0$, this implies that for all $z\geq 0$,
\[
    \sum_{j=0}^K \P\{J=j\} \cdot z^j
    \, \leq \, 
    \left(1-\frac{K}{N}+\frac{Kz}{N}\right)^K \, .
\]
Integrating over $z\in[0,r]$ for $r>0$ yields
\begin{align*}
    \sum_{j=0}^K \P\{J=j\}\,\frac{r^{j+1}}{j+1}
    &\leq
    \int_0^r
    \left(1-\frac{K}{N}+\frac{Kz}{N}\right)^K \d z
    =
    \frac{N}{K(K+1)}
    \left(1-\frac{K}{N}+\frac{Kz}{N}\right)^{K+1}\Bigg|_{z=0}^{r}
    \\
    &\leq
    \frac{N}{K^2}\left(1+\frac{Kr}{N}\right)^{K+1}.
\end{align*}
Dropping the $j=0$ term on the left, setting $r=1/\tau$, multiplying both sides by $2\tau^2$,
and using \eqref{eq:varQbnd} gives
\begin{align}
\Var_{\Q}(L)
&\leq
\sum_{j=1}^K \P\{J=j\}\,\frac{2}{(j+1)\tau^{j-1}}
\, \leq \, 
\frac{2N\tau^2}{K^2}
\left(1+\frac{K}{N\tau}\right)^{K+1}
\label{eq:varQL2_upper_bnd_unif_a}
\end{align}
Consequently, $\TV(\P,\Q)\le \frac12\sqrt{\Var_{\Q}(L)}\to 0$ whenever the right-hand side of
\eqref{eq:varQL2_upper_bnd_unif_a} tends to zero, and weak detection is impossible in that regime. This proves Theorem~\ref{thm:hard-cluster-flat}(b).

\section{The von Mises flat model} \label{sec-von-mises-flat}

\begin{theorem}
\label{thm:vm_flat}
Consider the von Mises flat model with parameters $(N,K,\kappa)$. For $\tau\in(0,1)$, define
\[
  p_\kappa(\tau) \triangleq \P\{\mathrm{vonMises}(0,\kappa)\in[-\pi\tau,\pi\tau]\} \, ,
  \qquad
  g \triangleq K\big(p_\kappa(\tau)-\tau\big) \, .
\]
\begin{itemize}
  \item[$\mathrm{(a)}$] (Achievability)
  Suppose there exist sequences $c_N\to\infty$ and $\gamma_N$ with
  \[
      \gamma_N \triangleq N\tau + g - c_N\sqrt{N\tau+g},
      ~~~~~~
      \gamma_N \geq 1+(N-1)\, \tau \ \text{for all sufficiently large }N,
  \]
  such that
  \begin{align}
    &\frac{g^2}{N\tau + g}  \to  \infty \, ,
    \label{eq:vm_suff_1}\\
    &\frac{\big(g - c_N\sqrt{N\tau + g}\big)^2}{2N\tau + g - c_N\sqrt{N\tau + g}}
      -  \log N \,\to\, \infty \, .
    \label{eq:vm_suff_2}
  \end{align}
  Then the interval test with window length $2\pi\tau$ and threshold $\gamma_N$ achieves strong detection.

  \item[$\mathrm{(b)}$] (Converse)
  Let $R(\kappa)\triangleq I_0(2\kappa)/I_0^2(\kappa)$, where $I_0$ is the modified Bessel function of the first kind. If
  \begin{align}
      \frac{K^2}{N}\big(R(\kappa)-1\big) - \log R(\kappa) \,\to\, -\infty,
      \label{eq:vm_conv}
  \end{align}
  then weak detection is impossible.
\end{itemize}
\end{theorem}

\bigskip

\begin{corollary}[Achievability regimes]
\label{cor:vm_flat_achievability}
Under the von Mises flat model, strong detection is possible in each of the following regimes with the interval test.
\begin{itemize}
  \item $K=N^{\alpha}$ with $0<\alpha\leq 1/2$, and
  \(
      \frac{1}{\sqrt{\kappa}}=\frac{K^2}{c_1\,N\log N}
  \)
  for a constant $c_1 > c_0$, where
  \[
      c_0 \triangleq \frac{2}{\pi}\min_{c_2>0}\frac{c_2}{\big(1-2 \, Q(c_2)\big)^2}\approx 0.5057,
      \qquad
      c_2^\star\approx 0.7518\ \text{(a minimizer)} \, ,
  \]    
  and $Q$ is the complementary CDF of the standard Gaussian. 
  Here, the test interval is chosen $\tau = {c_2^\star}/{(\pi\sqrt{\kappa}\,)}$, and the threshold $\gamma_N$ is chosen so that $c_N\to\infty$ and $c_N=o\big(\sqrt{\log N}\big) \, .
  $
  [For large $\kappa$, $\mathrm{vonMises}(0,\kappa)\approx \mathcal N(0,1/\kappa)$ so
  $p_\kappa(\tau)\to 1-2Q(c_2^\star)$ and $g\sim K(1-2Q(c_2^\star))$.
  With $\tau=c_2^\star/(\pi\sqrt{\kappa})$ and $1/\sqrt{\kappa}=K^2/(c_1N\log N)$, it follows that
  \[
      \frac{g^2}{2N\tau\log N}\sim \frac{c_1(1-2Q(c_2^\star))^2}{(2/\pi)c_2^\star}>1 \, ,
  \]
  and hence
  \eqref{eq:vm_suff_1}--\eqref{eq:vm_suff_2} are satisfied.]

  \item $K=N^{\alpha}$ with $1/2<\alpha<1$, and $\ \kappa=\Omega(1)$. Here, the test interval is chosen $\ \tau\in[\epsilon,1-\epsilon]$
  for some constant $\epsilon\in(0,1/2)$, and the threshold $\gamma_N$ is chosen so that $c_N\to\infty$ and $c_N=o\big(K/\sqrt{N}\big)$.
  [Then $p_\kappa(\tau)-\tau$ is bounded away from $0$, so $g=\Theta(K)$ and
  \eqref{eq:vm_suff_1}--\eqref{eq:vm_suff_2} hold.]
\end{itemize}
\end{corollary}

\bigskip 

\begin{corollary}[Impossibility regime]
\label{cor:vm_flat_impossibility}
Under the von Mises flat model, weak detection is impossible when
$K = N^{\alpha}$ for some constant $0<\alpha<1/2$, and 
$\frac{1}{\sqrt{\kappa}}=\frac{K^2}{c\,N\log N}$
for some constant  $c<\frac{1-2\alpha}{2\sqrt{\pi}} \, .$
[
    In this scaling we have $\kappa\to\infty$, hence $R(\kappa)=I_0(2\kappa)/I_0^2(\kappa)\sim \sqrt{\pi\kappa}$
    and $\log R(\kappa)\sim \frac12\log\kappa$. Therefore the converse condition \eqref{eq:vm_conv} is implied by
    \[
        \frac{K^2}{N}\sqrt{\pi\kappa} \, \leq \, \frac12\log\kappa-\omega(1).
    \]
    With $1/\sqrt{\kappa}=K^2/(cN\log N)$, the left-hand side equals $\sqrt{\pi}\,c^{-1}\log N$, while the right-hand side satisfies 
    \[
        \frac12\log\kappa=\log(cN\log N/K^2)=(1-2\alpha)\log N+O(\log\log N).
    \]    
    Thus, if $c<(1-2\alpha)/(2\sqrt{\pi})$, the inequality holds with an $\omega(1)$ slack, and \eqref{eq:vm_conv} follows.
]
\end{corollary}

\subsection{Achievability for von Mises flat model -- interval test}
\label{sec:proof_vm_flat}

Fix $\tau\in(0,1)$ and write $p_\kappa(\tau)$ and $g$ as in the theorem statement.
We apply the interval test with window length $2\pi\tau$ and threshold
\(
    \gamma_N \triangleq N\tau + g - c_N\sqrt{N\tau+g}.
\)

\emph{False alarm.}
Under $H_0$ the observations are i.i.d.\ uniform on $[0,2\pi)$, exactly as in the hard-cluster flat model.
Hence the false alarm probability of the interval test is identical to that case.
In particular, as in \eqref{eq:pfa_bnd},
whenever $\gamma_N \ge 1+(N-1)\tau$,
\begin{align} \label{eq:pfa_vm_repeat}
    p_{\fa}
    \, \leq \, 
    N\cdot \exp\left\{
    -\frac{\big(\gamma_N-1-(N-1)\tau\big)^2}{1+(N-1)\tau+\gamma_N}
    \right\}.
\end{align}
%
%
Noting that
\[
    \gamma_N-1-(N-1)\,\tau
    =
    g-c_N\sqrt{N\tau+g}-(1-\tau),
    ~~~~~~
    1+(N-1)\,\tau+\gamma_N
    =
    2N\tau+g-c_N\sqrt{N\tau+g} \,+\, O(1),
\]
we see that condition \eqref{eq:vm_suff_2} implies that the exponent in \eqref{eq:pfa_vm_repeat}
dominates $\log N$, and hence $p_{\fa}\to 0$.

\medskip
\emph{Missed detection.}
Under $H_1$, consider the test interval of length $2\pi\tau$ centered at the planted phase $\Theta^*$,
namely $[\Theta^*-\pi\tau,\Theta^*+\pi\tau]$ (mod $2\pi$).
Let $Y$ be the number of observations falling in this window.  Under $H_1$, $Y$ decomposes as
\[
    Y=Y_0+Y_1 \, ,
    ~~~~~~
    Y_0\sim \Binom\big(K,p_\kappa(\tau)\big) \, ,
    ~~~~~~
    Y_1\sim \Binom\big(N-K,\tau\big) \, ,
\]
with $Y_0$ and $Y_1$ independent.  In particular,
\[
    \mu_1 \triangleq \E_{\P}[Y] = (N-K)\tau + Kp_\kappa(\tau) = N\tau + g.
\]
Since the interval test rejects $H_0$ whenever $Y \geq \gamma_N$, we have
\[
    p_{\miss}
    \, \leq \,
    \P\{Y<\gamma_N\}
    \, = \,
    \P\{Y_0+Y_1<\gamma_N\} \, .
\]
Because $Y_0+Y_1$ is a sum of independent Bernoulli random variables, we may apply the standard Chernoff lower-tail bound (Poisson-binomial bound).  Writing $\gamma_N=(1-\eta)\mu_1$ yields
\[
    p_{\miss}
    \leq
    \exp\left(-\frac{\eta^2}{2}\mu_1\right)
    =
    \exp\left(
    -\frac{(\mu_1-\gamma_N)^2}{2\mu_1}
    \right).
\]
Substituting $\mu_1=N\tau+g$ and $\gamma_N=N\tau+g-c_N\sqrt{N\tau+g}$ gives
\begin{align} \label{eq:pmiss_vm}
    p_{\miss}
    \,\leq\,
    \exp\left(-\frac{c_N^2}{2}\right)
    \, \to \, 0 \,,
\end{align}
since $c_N\to\infty$.  Combining with $p_{\fa}\to 0$ proves strong detection. This proves part (a) of Theorem~\ref{thm:vm_flat}.

\medskip 

\subsection{Converse for von Mises flat model}

We bound $\Var_{\Q}(L)$, where $L=\frac{\d\P}{\d\Q}$. Since $\E_{\Q}[L] = 1$ and $\TV(\P,\Q) \leq \frac12\sqrt{\Var_{\Q}(L)}$,
it suffices to show that $\E_{\Q}[L^2]\to 1$.

For fixed $(S,\theta)$, the conditional likelihood ratio is
\begin{align} \label{eq:vm_LS}
    L_{S,\theta}(X)
    \, = \,
    I_0(\kappa)^{-K}\,
    \exp\left(\kappa\sum_{i\in S}\cos(X_i-\theta)\right).
\end{align}
Fix $S,\theta,S',\theta'$ with $|S|=|S'|=K$ and let $j=|S\cap S'|$.  Then
\begin{align}
    \E_{\Q}\big[L_{S,\theta}(X)L_{S',\theta'}(X)\big]
    & =
    I_0(\kappa)^{-2K}\cdot 
    \E_{\Q}\Big[
        \exp\Big(
        \kappa\sum_{i\in S}\cos(X_i-\theta)
        +
        \kappa\sum_{i\in S'}\cos(X_i-\theta')
        \Big)
    \Big]\nonumber\\
    &=
    I_0(\kappa)^{-2K}\cdot 
    \E_{\Q} \Big[
        \exp \Big(
        \kappa\sum_{i\in S\cap S'}\big(\cos(X_i-\theta)+\cos(X_i-\theta')\big)
        \Big)
    \Big]\label{eq:vm_step_a}\\
    &\hspace{1.86cm}\times
    \E_{\Q} \Big[
        \exp \Big(
        \kappa\sum_{i\in S\setminus S'}\cos(X_i-\theta)
        +
        \kappa\sum_{i\in S'\setminus S}\cos(X_i-\theta')
        \Big)
    \Big]\nonumber\\
    &=
    I_0(\kappa)^{-2j}\cdot 
    \E_{\Q}\left[
        \exp\left(
        \kappa\big(\cos(X_1-\theta)+\cos(X_1-\theta')\big)
        \right)
    \right]^j,\nonumber
    \end{align}
where in \eqref{eq:vm_step_a} we used independence under $\Q$, and in the last step we used that the
$X_i$ are i.i.d.\ uniform on $[0,2\pi)$, so each of the $j$ intersection terms has the same expectation.

Using the identity $\cos a+\cos b = 2\,\cos\big(\frac{a-b}{2}\big)\,\cos\big(\frac{a+b}{2}\big)$, we get
\[
    \cos(X_1-\theta)+\cos(X_1-\theta')
    =
    2\cos\Big(\frac{\theta-\theta'}{2}\Big)
    \cdot \cos\Big(X_1-\frac{\theta+\theta'}{2}\Big) \, ,
\]
and therefore
\begin{align*}
    \E_{\Q}\big[
        \exp\left(
        \kappa\big(\cos(X_1-\theta)+\cos(X_1-\theta')\big)
        \right)
    \big]
    &=
    \E_{\Q}\Big[
        \exp\Big(
        2\kappa\cos\Big(\frac{\theta-\theta'}{2}\Big) \cdot 
        \cos\Big(X_1-\frac{\theta+\theta'}{2}\Big) \Big)
    \Big]\\
    &=
    I_0\Big(2\kappa\cos\Big(\frac{\theta-\theta'}{2}\Big)\Big) \, ,
\end{align*}
since $X_1-(\theta+\theta')/2$ is uniform on $[0,2\pi)$ under $\Q$.
Hence
\begin{align}
\label{eq:vm_Eprod}
    \E_{\Q}\Big[L_{S,\theta}(X)L_{S',\theta'}(X)\Big]
    =
    \Bigg[
        \frac{I_0\left(2\kappa\cos\Big(\frac{\theta-\theta'}{2}\right)\Big)}
        {I_0^2(\kappa)}
    \Bigg]^j.
\end{align}
To compute $\E_{\Q}[L^2]$, we view $S,\theta,S',\theta'$ as independent and uniform over their domains. Letting $\vartheta\triangleq ({\theta-\theta'})/{2}$, we may write
\begin{align}
    \E_{\Q}[L^2]
    &=
    \frac{1}{2\pi}\int_0^{2\pi}
    \sum_{j=0}^K
    \P\{|S\cap S'|=j\}
    \left[\frac{I_0(2\kappa\cos\vartheta)}{I_0^2(\kappa)}\right]^j
    \,\d\vartheta.
    \label{eq:EQL2_start}
\end{align}
We bound the sum over $j$ using the fact that $|S\cap S'|$ is hypergeometric$(N,K,K)$ and is
second-order dominated by $\Binom\left(K,\frac{K}{N}\right)$ (Hoeffding~\cite{hoeffding1963probability}).
Since $z\mapsto z^j$ is convex for $z\geq 0$, this implies that for all $z\geq 0$,
\[
    \sum_{j=0}^K \P\{|S\cap S'|=j\} \cdot z^j
    \, \leq \,
    \sum_{j=0}^K
    \P\left\{\Binom \left(K,\frac{K}{N}\right)=j\right\} \cdot z^j
    \, = \,
    \left(1-\frac{K}{N}+\frac{Kz}{N}\right)^K.
\]
Applying this with $z=\rho_\kappa(\vartheta)$, where
\[
    \rho_\kappa(\vartheta)\triangleq \frac{I_0(2\kappa\cos\vartheta)}{I_0^2(\kappa)} \, ,
\]
and then using $1+x\leq e^x$ yields
\begin{align}
    \E_{\Q}[L^2]
    &\leq
    \frac{1}{2\pi}\int_0^{2\pi}
    \left(1+\frac{K}{N}\big(\rho_\kappa(\vartheta)-1\big)\right)^K
    \,\d\vartheta \nonumber\\
    &\leq
    \frac{1}{2\pi}\int_0^{2\pi}
    \exp\left(\frac{K^2}{N}\big(\rho_\kappa(\vartheta)-1\big)\right)
    \,\d\vartheta
    =
    e^{-\frac{K^2}{N}}\cdot
    \frac{1}{2\pi}\int_0^{2\pi}
    \exp\left(\frac{K^2}{N}\rho_\kappa(\vartheta)\right)
    \,\d\vartheta \, .
    \label{eq:EQL2_half}
\end{align}
Let $R(\kappa)\triangleq \max_{\vartheta}\rho_\kappa(\vartheta)=\rho_\kappa(0)=\frac{I_0(2\kappa)}{I_0^2(\kappa)}$.
For $x\in[0,R(\kappa)]$ and $a>0$, convexity of $x\mapsto e^{ax}$ on $[0,R(\kappa)]$ implies
\[
    e^{ax}
    \leq
    \left(1-\frac{x}{R(\kappa)}\right)e^{0}
    +\frac{x}{R(\kappa)}e^{aR(\kappa)}
    =
    1+\frac{e^{aR(\kappa)}-1}{R(\kappa)}\,x \, .
\]
Applying this with $a=K^2/N$ and $x=\rho_\kappa(\vartheta)$ in \eqref{eq:EQL2_half} gives
\begin{align}
    \E_{\Q}[L^2]
    &\leq
    e^{-\frac{K^2}{N}}\cdot
    \frac{1}{2\pi}\int_0^{2\pi}
    \left(
        1+\frac{e^{\frac{K^2}{N}R(\kappa)}-1}{R(\kappa)}\,\rho_\kappa(\vartheta)
    \right)
    \,\d\vartheta \, .
\label{eq:EQL2_two_thirds}
\end{align}

\begin{lemma}
\label{lem:vm_rho_mean_one}
For every $\kappa>0$,
\[
    \frac{1}{2\pi}\int_0^{2\pi}\rho_\kappa(\vartheta)\,\d\vartheta = 1 \, .
\]
\end{lemma}

\begin{proof}
Using $2\cos(a)\cos(b)=\cos(a+b)+\cos(a-b)$, we have
\begin{align*}
    \frac{1}{2\pi}\int_0^{2\pi} I_0 \big(2\kappa\cos\vartheta\big)\,\d\vartheta
    &=
    \frac{1}{(2\pi)^2}\int_0^{2\pi}\int_0^{2\pi}
    \exp\Big(2\kappa\cos\vartheta\cos(\varphi-\vartheta)\Big)
    \,\d\varphi\,\d\vartheta
    \\
    &=
    \frac{1}{(2\pi)^2}\int_0^{2\pi}\int_0^{2\pi}
    \exp\Big(\kappa\cos\varphi+\kappa\cos(2\vartheta-\varphi)\Big)
    \,\d\vartheta\,\d\varphi
    \\
    &=
    I_0^2(\kappa) \, ,
\end{align*}
where in the last step we used the integral representation of $I_0(\kappa)$ twice.
Dividing both sides by $I_0^2(\kappa)$ proves the claim.
\end{proof}

Applying Lemma~\ref{lem:vm_rho_mean_one} to \eqref{eq:EQL2_two_thirds} yields
\begin{align*}
    \E_{\Q}[L^2]
    & \,\leq \,
    e^{-\frac{K^2}{N}}
    \left(
    1+\frac{e^{\frac{K^2}{N}R(\kappa)}-1}{R(\kappa)}
    \right)
    \,=\,
    \frac{e^{\frac{K^2}{N}(R(\kappa)-1)}}{R(\kappa)}
    +\left(1-\frac{e^{-\frac{K^2}{N}}}{R(\kappa)}\right).
\end{align*}
Therefore a sufficient condition for $\E_{\Q}[L^2]\to 1$ (and hence $\Var_{\Q}(L)\to 0$ and
$\TV(\P,\Q)\to 0$) is
\[
    \frac{K^2}{N}\big(R(\kappa)-1\big)-\log R(\kappa)\to -\infty \, ,
\]
which is exactly \eqref{eq:vm_conv}. This proves part (b) of Theorem~\ref{thm:vm_flat}.

\section{The hard-cluster community model} \label{sec-hard-cluster-community}

\begin{theorem} \label{thm:hard_cluster_community}
Consider the hard-cluster community model with parameters $(n,k,\tau)$.
\begin{itemize}
    \item[$\mathrm{(a)}$] (Achievability)
    The interval test with window length $2\pi\tau$ and threshold $k$ yields $p_{\miss}=0$ and
\begin{align}
p_{\fa} \leq \exp\left(\log \frac { \binom k 2 } {\tau} + k\left\{ \log\frac n k +1 + \frac{k-1} 2\log \tau \right\}  \right).
\label{eq:pfa_hardc_comm_interval}
\end{align}
Thus, strong detection is achieved if the right hand side of \eqref{eq:pfa_hardc_comm_interval} converges to zero.
    \item[$\mathrm{(b)}$] (Converse)
    If
    \begin{align}
        \label{eq:uniform_comm_converse_condition}
        \frac{k^2}{n}\left(\left(\frac{1}{\tau}\right)^{\frac{k-1}{2}}-1\right) \to 0
        \qquad\text{as } n\to\infty \, ,
    \end{align}
    then weak detection is impossible.
\end{itemize}
\end{theorem}

\bigskip 

\begin{corollary}[Achievability regimes]
\label{cor:uniform_comm_achievability}
Under the hard-cluster community model, each of the following conditions implies strong detection by the interval test:
\begin{itemize}
    \item $k=c\log n$ for a constant $c>0$ and $\ \tau \leq e^{-2/c} \, .$

    \item $3/\epsilon < k \leq n$ and $\ \tau \le \left(\frac{k}{ne}\right)^{\frac{2+\epsilon}{k-1}}$
    for a fixed $\epsilon\in (0,1) \, .$  

         \item  $k=\omega(\log n)$ and $\tau \leq 1 - \frac{2(1+\epsilon)}{k-1}\log \frac{ne}{k}$  
    
    \item $3\le k\le n$ and $\tau = \left( \frac k {ne} \right)^{\frac{(1+\epsilon)k}{\binom k 2 -1}}$ for a fixed  $\epsilon>0 \, .$

     \item   $3/\epsilon <  k = o(n)$ and $\tau =  \tau_{\epsilon} \triangleq \left( \frac k n \right)^{\frac{2+\epsilon}{k-1}}$ for some $\epsilon\in (0,1) \, .$
\end{itemize}
\end{corollary} 

\begin{proof}[Proof of Corollary \ref{cor:uniform_comm_achievability}]
For $k=c\log n$ and $\tau = e^{-2/c}$ the bound
\eqref{eq:pfa_hardc_comm_interval} becomes
\begin{align*}
    p_{\fa} & \leq \exp\left(\log \frac { \binom k 2}  {\tau} + c\log n \left\{ \log\frac n k +1 + \frac{c\log n-1} 2\left(-\frac 2 c\right) \right\}  \right) \\
    & = \frac 1 \tau \exp\left( \log\binom k 2 + (\log n )(-c \log k + c +1)   \right) \to 0 \, .
\end{align*}
For  $\tau = \left( \frac k {ne} \right)^{\frac{2+\epsilon}{k-1}}$  the bound \eqref{eq:pfa_hardc_comm_interval} becomes:
\begin{align*} 
    p_{\fa} & \leq \exp\left(\frac{2+\epsilon}{k-1} \log \frac n k + \log  \binom k 2  + k\left\{ -\frac {\epsilon} 2\log\frac n k -\frac {\epsilon} 2 \right\}  \right) \\
    & = \exp\left(\left[-\frac{k\epsilon} 2 +\frac{2+\epsilon}{k-1} \right]\log \frac n k + \log  \binom k 2  - \frac{\epsilon k} 2   \right) \\
    & \to 0 \mbox{ if } 3/\epsilon < k \leq n \, .
\end{align*}
For $\tau \leq 1 - \frac{2(1+\epsilon)}{k-1}\log \frac{ne}{k}$ and $k=\omega(\log n)$, since
\[
    \frac{2(1+\epsilon)}{k-1}\log \frac{ne}{k} \to 0 \, , ~~\mbox{and } ~~
    \frac{2(1+\epsilon)}{k-1}\log \frac{ne}{k} = \left(\frac{1+\epsilon}{1 + \epsilon/2 }\right)\frac{2+\epsilon}{k-1}\log \frac{ne}{k}\,,
\] 
it follows that $1 - \frac{2(1+\epsilon)}{k-1}\log \frac{ne}{k} \leq  \left(\frac{k}{ne}\right)^{\frac{2+\epsilon}{k-1}}$ so $p_{\fa}\to 0$ by the previous case.

For  $\tau =  \left( \frac k {ne} \right)^{\frac{(1+\epsilon)k}{\binom k 2 -1}}$  and $3\leq k \leq n$ the bound \eqref{eq:pfa_hardc_comm_interval} becomes:
\begin{align*} 
    p_{\fa} & \leq \exp\left( \log  \binom k 2  - \epsilon k \log \frac{ne} k  \right) \to 0 \, . 
\end{align*}
For  $\tau =  \tau_{\epsilon} \triangleq 
\left( \frac k n \right)^{\frac{2+\epsilon}{k-1}}$  the bound \eqref{eq:pfa_hardc_comm_interval} becomes:
\begin{align*} 
    p_{\fa} & \leq \exp\left(\frac{2+\epsilon}{k-1} \log \frac n k + \log  \binom k 2  + k\left\{ -\frac {\epsilon} 2\log\frac n k +1 \right\}  \right) \\
    & = \exp\left(\left[-\frac{k\epsilon} 2 +\frac{2+\epsilon}{k-1} \right]\log \frac n k + \log  \binom k 2  + k   \right) \\
    & \to 0 \mbox{ if } k = o(n) \mbox{ and $k> 3/\epsilon $} \, .
    \qedhere
\end{align*}
\end{proof}

\bigskip 

\begin{corollary}[Impossibility regimes]
\label{cor:uniform_comm_impossibility}
Under the hard-cluster community model, each of the following conditions implies the impossibility of weak detection:
\begin{itemize}
    \item $k=c\log n$ for a constant $c>0$ and a constant $\tau>e^{-2/c}$;

    \item $k=o(\sqrt{n})$ and $\ \tau \ge \exp\!\left(-\frac{(2-\epsilon)\log(n/k^2)}{k-1}\right)$
    for a constant $\epsilon>0$;

    \item $k = \omega(\log n)$ and $\tau \geq 1 - \frac{2(1-\epsilon)}{k-1}\log\left(1 + \frac{\epsilon_n n}{k^2}\right)$ for positive constants $\epsilon$ and $\epsilon_n$ with $\epsilon_n\to 0$; [These conditions imply $\tau \to 1$ so $\tau \geq \exp\left\{ \frac{2}{k-1}\log\left(1 + \frac{\epsilon_n n}{k^2}\right)\right\}$ for sufficiently large $n,$ where $\exp\left\{ \frac{2}{k-1}\log\left(1 + \frac{\epsilon_n n}{k^2}\right)\right\}$ is the value of $\tau$ obtained by setting the
     left hand side of \eqref{eq:uniform_comm_converse_condition} to $\epsilon_n.$]

\end{itemize}
\end{corollary}

\subsection{Achievability for hard-cluster community model -- interval test}
\label{sec:uniform_community_converse}

Under $H_1$, for the planted pair $(C^*,\Theta^*)$, we have
\(
    X_{ij}\in[\Theta^*,\Theta^*+2\pi\tau]
\)
for all $i,j\in C^*$, $i\neq j$. Hence the test succeeds with $\theta=\Theta^*$ and $C=C^*$, so
\(
    p_{\miss}=0\, .
\)

We now compute the false alarm probability for the interval test with $\ell = 2\pi\tau$. Under $H_0$, the edge labels $(X_{ij})_{1\leq i<j\leq n}$ are i.i.d.\ uniform on $[0,2\pi)$.
For a fixed subset $C\subset[n]$ with $|C|=k$, let
\[
    E_C \,\triangleq\,
    \Big\{\exists\,\theta\in[0,2\pi):\ X_{ij}\in[\theta,\theta+2\pi\tau]\ \ \forall\, i,j\in C,\ i\neq j\Big\}.
\]
By a union bound over all communities,
\begin{align}\label{eq:pfa_union_C}
    p_{\fa} \,=\, \Q\,\Big(\bigcup_{C:|C|=k} E_C\Big)\,\le\, \binom{n}{k}\,\Q(E_C),
\end{align}
where $\Q(E_C)$ does not depend on the particular choice of $C$.

To bound $\Q(E_C)$ while scanning over $\theta$, note that if $E_C$ occurs, then in particular all
\(
\binom{k}{2}
\)
angles $\{X_{ij}: i<j,\ i,j\in C\}$ lie in some arc of length $2\pi\tau$.
Let $e^\star$ denote an edge in $C$ attaining the minimum angle among these $m$ angles.
Then \(E_C\) implies that all remaining \(\binom k 2 - 1\) angles lie in the arc
\([X_{e^\star},X_{e^\star}+2\pi\tau]\).
Taking a union bound over the choice of the minimizing edge yields
$\Q(E_C)\leq \binom k 2 \tau^{\binom k 2 -1}.$
 Therefore,
\begin{align*} 
p_{\fa} & \leq  \binom k 2 \binom{n}{k} \tau^{\binom{k}{2}-1} 
  \leq   \binom k 2  \left( \frac{ne} k \right)^k 
                      \tau^{\binom{k}{2}-1}  \nonumber\\
            & = \exp\left(\log \frac { \binom k 2 } {\tau} + k\left\{ \log\frac n k +1 + \frac{k-1} 2\log \tau \right\}  \right).
\end{align*}
\noindent
This proves Theorem~\ref{thm:hard_cluster_community}(a).

\subsection{Converse for hard-cluster community model}
We bound $\Var_{\Q}(L)$ above, where $L(X)=\frac{\d\P}{\d\Q}(X)$ and
$X=(X_{ij})_{1\le i<j\le n}$. Since $\P$ and $\Q$ are absolutely continuous,
$\TV(\P,\Q) \le \frac{1}{2}\sqrt{\Var_{\Q}(L)}$, so if $\Var_{\Q}(L)\to 0$ (equivalently $\E_{\Q}[L^2]\to 1$),
then $\TV(\P,\Q)\to 0$ and weak detection is impossible.

Averaging over the possible values $(C,\theta)$ of $(C^*,\Theta^*)$ we find
\[
    L(X)=\int_{0}^{2\pi}\frac{1}{\binom{n}{k}}
    \sum_{C\subset[n]:|C|=k} L_{C,\theta}(X)\,\frac{\d\theta}{2\pi},
\]
where
\[
    L_{C,\theta}(X)=
    \left\{
    \begin{array}{cl}
        \tau^{-\binom{k}{2}}
        & \mbox{if } X_{ij}\in[\theta,\theta+2\pi\tau] \mbox{ for all } i,j\in C,\ i<j,\\
        0 & \mbox{otherwise.}
    \end{array}
    \right.
\]
For $C,\theta,C',\theta'$ fixed, let $S=|C\cap C'|$ and
\[
    \delta \triangleq \frac{1}{2\pi}\,\mbox{length}\big([\theta,\theta+2\pi\tau]\cap[\theta',\theta'+2\pi\tau]\big) \, .
\]
Then
\begin{align*}
    \E_{\Q}\!\big[L_{C,\theta}(X)L_{C',\theta'}(X)\big]
    &=
    \frac{1}{\tau^{2\binom{k}{2}}}\, 
    \Q\left\{
    \begin{array}{l}
        X_{ij}\in[\theta,\theta+2\pi\tau]\cap[\theta',\theta'+2\pi\tau]
        \ \text{ for } i,j\in C\cap C',\ i<j,\\
        X_{ij}\in[\theta,\theta+2\pi\tau]
        \ \text{ for } i,j\in C,\ \{i,j\}\not\subset C\cap C',\ i<j,\\
        X_{ij}\in[\theta',\theta'+2\pi\tau]
        \ \text{ for } i,j\in C',\ \{i,j\}\not\subset C\cap C',\ i<j
    \end{array}\right\} \\
    &=
    \frac{\delta^{\binom{S}{2}}\tau^{k(k-1)-S(S-1)}}{\tau^{k(k-1)}}
    = \delta^{\binom{S}{2}}\,\tau^{-2\binom{S}{2}}.
\end{align*}
Therefore
\begin{align*}
    \E_{\Q}[L^2]
    &= \E_{C,C',\Theta,\Theta'}\!\left[\E_{\Q}\big[L_{C,\Theta}(X)L_{C',\Theta'}(X)\big]\right] \\
    &= \E_S\!\left[\E_{\Theta,\Theta'}\!\left[\delta^{\binom{S}{2}}\tau^{-2\binom{S}{2}}\right]\right]
    = \E_S\!\left[\E_{\Theta,\Theta'}\!\left[\left(\frac{\delta}{\tau^2}\right)^{\binom{S}{2}}\right]\right].
\end{align*}
Since the maximum overlap length is $2\pi\tau$, we have $\delta\le \tau$, hence $\delta/\tau^2\le 1/\tau$ and
\begin{align*}
    \E_{\Q}[L^2]
    \leq \E_S \bigg[\bigg(\frac{1}{\tau}\bigg)^{\binom{S}{2}} \, \bigg]
    \, \stepa{=} \, 
    \E_S \bigg[ \left(\frac{1}{\tau}\right)^{\frac{k-1}{2} \times S} \, \bigg] \, ,
\end{align*}
where (a) used $\binom{S}{2}=\frac{S(S-1)}{2} \leq \frac{k-1}{2}\,S$ for $0\leq S\leq k$.

The variable $S$ is distributed as $\mathrm{hypergeometric}(n,k,k)$, which is dominated in the convex order by
$\Binom(k,p)$ with $p=k/n$. Let $z=(1/\tau)^{\frac{k-1}{2}}$. Using that $s\mapsto z^s$ is convex for $z\ge 0$,
\[
    \E_S[z^S]\le \left(1-p+pz\right)^k
    = \left(1+\frac{k}{n}(z-1)\right)^k.
\]
Finally, since $1+x\le e^x$,
\[
    \E_{\Q}[L^2]
    \le \exp\!\left(\frac{k^2}{n}(z-1)\right)
    = \exp\!\left(\frac{k^2}{n}\left(\left(\frac{1}{\tau}\right)^{\frac{k-1}{2}}-1\right)\right).
\]
If \eqref{eq:uniform_comm_converse_condition} holds, the exponent tends to $0$, hence $\E_{\Q}[L^2]\to 1$ and
$\Var_{\Q}(L)=\E_{\Q}[L^2]-1\to 0$, which implies $\TV(\P,\Q)\to 0$ and completes the proof of Theorem~\ref{thm:hard_cluster_community}(b).

\section{The von Mises community model} \label{sec-von-mises-community}

Recall the functions
\[ 
    R(\kappa) = \frac{I_0(2\kappa)}{I_0^2(\kappa)} \, ~~~~ ~~~~ A(\kappa) = \frac{I_1(\kappa)}{I_0(\kappa)} \,,
\]
where $I_0$ and $I_1$ are the modified Bessel functions of the first kind, of order $0$ and $1$ respectively. See Appendix~\ref{app:function_bounds} for definitions and properties of these functions.

\medskip

\begin{theorem} \label{thm:vonmises-community}
Consider the von Mises community model with parameters $(n,k,\kappa).$ 
\begin{itemize}
    \item[$\mathrm{(a)}$] (Achievability)
    The interval test with window length $2\pi\tau$ and threshold $k$ satisfies
    \begin{align}
        p_{\fa} &\leq \exp\bigg(\log \frac { \binom k 2 } {\tau} + k \cdot \bigg\{ \log\frac n k +1 + \frac{k-1} 2\log \tau \bigg\}  \bigg) \, ,
        \label{eq:pfa-vM-community} 
        \\[2pt]
         p_{\miss} & \leq \binom k2 \cdot \frac{\exp\{ (\cos(\pi \tau) - 1) \kappa \}}{2\pi^2 \, I_0(\kappa) \, e^{-\kappa} \cdot \kappa \cdot |\sin(\pi \tau)|} \, .
        \label{eq:pmiss-vM-community}
    \end{align}
    Thus, strong detection is achieved by the interval test if the right hand sides of~\eqref{eq:pfa-vM-community} and~\eqref{eq:pmiss-vM-community} converge to $0$.
    Moreover, the coherence test for any integer $B\geq 3$, $\epsilon \in (0,1)$ and suitable threshold $\beta$ satisfies
    \begin{align} 
        p_\fa & \, \leq \, B \cdot \exp \left\{k\left( \log\left( \frac{ne}{k}\right)-(1-\epsilon/4)^2 \cdot (k-1) \cdot A^2(\kappa) \cdot \cos^2(\pi/B)/2\right) \right\} \, , \label{eq: pfa-vM-community-2}
        \\[2pt]
        p_\miss &  \, \leq  \, \exp\left(- \frac{\epsilon^2 K \, A^2(\kappa)}{32}\right) \, . \label{eq: pmiss-vM-community-3}
    \end{align}
    Thus, strong detection is achieved by the coherence test if the right hand sides of~\eqref{eq: pfa-vM-community-2} and~\eqref{eq: pmiss-vM-community-3} converge to $0$.
    \item[$\mathrm{(b)}$] (Converse)
    Weak detection is impossible if
    \begin{align}
        \label{eq:matrix-necessary-general}
        \frac{k^2}{n}\left(\exp\left(\frac{(k-1)\log R(\kappa)}{2}\right)-1\right)\ \xrightarrow[n\to\infty]{}\ 0 \, .
    \end{align}
\end{itemize}
\end{theorem}

\bigskip 

\begin{corollary}[Achievability regimes] \label{cor-vM-community-achievability}
Under the von Mises community model, strong detection is possible in each of the following regimes.
\begin{itemize}
    \item $3/\epsilon < k=o(\log n)$ for some $\epsilon \in(0,1)$ and $\kappa \, \geq\, \Big(\frac{4}{\pi^2}+\epsilon\Big)\,(\log k)\,\Big(\frac{n}{k}\Big)^{\frac{4(1+\epsilon)}{k-1}} \, .$
    [Use the interval test.
        Setting
        \(
            \tau \equiv \tau_\epsilon \triangleq \big({k}/{n}\big)^{\frac{2+\epsilon}{k-1}} 
        \)
        and using the interval test with window length $2\pi\tau$ and subset size $k$, we have from~\eqref{eq:pfa-vM-community} that $p_\fa\to 0$ (see proof of Corollary~\ref{cor:uniform_comm_achievability}). 
        Further, using that $\tau \to 0$, $\kappa \to \infty$, $\cos(\pi \tau)-1 = -\frac{(\pi\tau)^2} 2(1+o(1)),$ $\sin(\pi\tau) = \pi\tau (1+ o(1)),$  and $I_0(\kappa)\sim \frac{e^{\kappa}}{\sqrt{2\pi \kappa}}$, the equation~\eqref{eq:pmiss-vM-community} yields
        \begin{align} \label{eq:pmiss-vM-community-cor}
            \log p_{\miss} \leq (\const) + 2\log k - \frac 1 2 \log (\kappa\tau^2)
            -\frac{\pi^2 \kappa\tau^2(1+o(1))} 2 \, ,
        \end{align}
        Therefore, since
        $\kappa\tau^2 \geq (\frac{4}{\pi^2}+\epsilon)(\log k)\,(\frac{n}{k})^{\frac{2\epsilon}{k-1}}$,
        the last (negative) term in~\eqref{eq:pmiss-vM-community-cor} dominates and $\log p_{\miss}\to -\infty$.
    ]
    \item $k=c\log n$ for some $c>2$ and fixed $\kappa$ with $A^2(\kappa) \, c > 2.$ [Use the coherence test. Take $B$ sufficiently large and $\epsilon$ sufficiently small.]
    \item $k = c\log n$ and $\kappa \geq \frac{2\log k}{1-\cos(\pi e^{-2/c})}$. [Set $\tau = e^{-2/c}$ and use the interval test with window size $2 \pi \tau$. Note $p_{\fa} \to 0$ for this choice of $k$ and $\tau$ (see proof of Corollary~\ref{cor:uniform_comm_achievability}). Using $I_0(\kappa)\sim \frac{e^{\kappa}}{\sqrt{2\pi \kappa}}$ (because $\kappa \to \infty$),~\eqref{eq:pmiss-vM-community} yields
    \(
        \log p_{\miss} \leq (\const) + 2\log k - \frac 1 2 \log (\kappa)
    +(\cos(\pi \tau) - 1)\, \kappa  \to -\infty.
    \)  
    ]
    \item $k = \omega(\log n)$ and $\kappa \geq (1+\epsilon)\sqrt{\frac{8\log n} {k-1}}$ for some $\epsilon$ with $0<\epsilon < 1.$  [Use the coherence test.  The righthand sides of \eqref{eq: pmiss-vM-community-3} and \eqref{eq: pfa-vM-community-2} are monotone decreasing in $\kappa$ so we may assume
    $\kappa = (1+\epsilon)\sqrt{\frac{8\log n} {k-1}}.$  The conditions imply $\kappa\to 0$ so that $A(\kappa) \sim \frac{\kappa} 2$ (see the appendix).  Since $(1-\epsilon/4)(1+\epsilon) \geq 1 + \epsilon/2$ we can select $B$ large enough that $(1-\epsilon/4)(1+\epsilon)\cos \big( \pi/ B \big) \geq 1 + \epsilon / 4.$ Then the exponent in \eqref{eq: pfa-vM-community-2} is less than or equal to $k\left\{ \log\left( \frac{ne} k\right) - (1+\epsilon/4)^2 \log n \right\}$
    so that $p_{\fa}\to 0.$   Also, $K\kappa^2 \to \infty$ so $p_{\miss} \to 0.$]
\end{itemize}
\end{corollary}

\bigskip 

\begin{corollary}[Impossibility regime] \label{cor-vM-community-impossibility}
Under the von Mises community model, each of the following conditions implies the impossibility of weak detection.
Note that they require $\kappa$ to be smaller for larger $k$ and in particular $\kappa\to 0$ for $k = \omega(\log n).$
\begin{itemize}
    \item $k=o(\log n)$ and $\kappa \leq  \left({n}/{k^2}\right)^{\frac{4-\epsilon}{k-1}}$ for some $\epsilon > 0.$
    [
        These conditions imply $\kappa\to\infty$ so $R(\kappa) \sim \sqrt{\pi \kappa}.$ The conditions also imply 
        \(
            \log\frac{k^2} n + \frac{(k-1)\log(\sqrt{\pi \kappa} (1+o(1))) } 2 \to -\infty \, .
        \)
    ]
    \item $k = c\log n$ for some $c>0$ and $c \log R(\kappa) < 2.$  
    [
        This implies $\log\frac{k^2} n + \frac{(k-1)\log(R(\kappa))} 2 \to -\infty.$
    ]
    \item  $k = \omega(\log n)$ and $k\leq n^c$ for some $c$ with $0 < c < 0.5$
    and $\kappa^2 \leq  \frac{4(1-2c-\epsilon) \log n}{k-1}$ for some $\epsilon > 0.$
    [
        These conditions imply $\kappa \to 0$ so that $\log R(\kappa)\sim \frac{\kappa^2} 2$  and they also imply $ \log\frac{k^2} n + \frac{(k-1)\kappa^2(1+o(1))}{4}  \to -\infty .$ 
    ]
\end{itemize}    
\end{corollary}

\medskip 

\begin{remark}
The first condition of Corollary~\ref{cor-vM-community-impossibility} matches the first sufficient condition for strong recovery in Corollary~\ref{cor-vM-community-achievability}, up to the $\log k$ factor. Moreover, if we strengthen the condition on $k$ a bit more to $k = o\left(\frac{\log n}{\log\log n}\right)$ then $\log k = o\left( \left( \frac n {k^2} \right)^\epsilon   \right)$ in which case the factor $\log k$ can be absorbed into the other term and is thus not meaningful.  
The second condition of Corollary~\ref{cor-vM-community-impossibility} lines up well with the second sufficient condition for strong recovery in Corollary~\ref{cor-vM-community-achievability} -- there is a factor of $2$ difference for $\kappa^2$ for the second condition and large values of $c$. In particular, as $\kappa\to 0$, $\log R(\kappa) \sim \frac{\kappa^2} 2$ and $A^2(\kappa)\sim \frac{\kappa^2} 4.$ Similarly, the third condition of Corollary~\ref{cor-vM-community-impossibility} lines up well with the fourth sufficient condition for strong recovery in Corollary~\ref{cor-vM-community-achievability}; there is again a factor of $2$ difference for the value of $\kappa^2$ when $c$ is close to $0$.
\end{remark}

\subsection{Achievability for von Mises community model -- interval test}

Under $H_0$, the edge variables are i.i.d.\ uniform on $[0,2\pi)$, exactly as in the hard-cluster community model. Therefore, the false-alarm bound
\eqref{eq:pfa-vM-community} is the same as~\eqref{eq:pfa_hardc_comm_interval}, already proven.

The interval detector returns decision $H_1$ if the labels of all intracommunity edges fall into the detection window centered at the true value of $\Theta.$  So $p_{\miss}$ is less than or equal to the probability under $\P$ that at least one of the $\binom {k}{2}$ such labels does not fall into the detection interval centered at the true value of $\Theta.$  So by the union bound,
\begin{align*}
    p_{\miss} \leq \binom {k}{2}(1-p_{\kappa}(\tau)) \, ,
\end{align*}
where $p_{\kappa}$ is defined in Theorem~\ref{thm:vm_flat}.
%

Using the fact that if $\Theta$ is uniformly distributed over $[0,2\pi]$ then $\cos \Theta$ has pdf $\frac 1 {\pi\sqrt{1-t^2}}$ for $t\in [-1,1]$, we find
\begin{align*}
    1-p_{\kappa}(\tau) 
    & \,=\, \frac 1 {2\pi I_0 (\kappa)} \int_{ \left\{\theta\in [0,2\pi] : \,\cos \theta \,\leq\, \cos \left(\pi \tau \right) \right\} }   e^{\kappa \cos \theta} \, \d\theta 
    \\[2pt]
    & \,=\,
    \frac 1 {2\pi^2 I_0(\kappa)} \, \int_{-1}^{\cos(\pi \tau)} \frac{e^{\kappa t}}{\sqrt{1-t^2}} \, \d t 
    \\[2pt]
    & \, \leq \, \frac 1 {2\pi^2 \,  I_0(\kappa) \, \sqrt{1-\cos^2(\pi \tau)}} \, \int_{-\infty}^{\cos(\pi \tau)} e^{\kappa t } \, \d t  \\[2pt]
    & \,=\, \frac{\exp\{ (\cos(\pi \tau) - 1) \kappa \}}{2\pi^2 \, I_0(\kappa) \, e^{-\kappa} \cdot \kappa \cdot |\sin(\pi \tau)|} 
\end{align*}
%
which establishes~\eqref{eq:pmiss-vM-community}. 

\medskip 

\begin{remark} Consider the interval test. If instead of using window size $2\pi \tau_{\epsilon}$ we were to use window size $2\pi\tau_f$
where 
\[
    \tau_f = \left( \frac k {ne} \right)^{\frac{(1+\epsilon)k}{\binom k 2 -1}}
\]
and if
\[
    \kappa \geq \tau_f^{-2}\left(\frac 4 {\pi^2} + \epsilon\right) (\log k)\left(\frac n k \right)^{\frac{2\epsilon}{k-1}}
\]
then strong recovery would be achieved for $k=o(\log n)$ (i.e. the condition $k\geq 3/\epsilon$ would not be needed).
\end{remark}

\subsection{Achievability for von Mises community model -- coherence test}

We bound $p_{\fa}$. By a union bound, we have for any $C\subset [n]$ with $|C|=k,$
\begin{align} \label{eq: pfa-1}
    p_{\fa} \leq \binom n k \cdot  \Q \bigg(  \bigg| \sum_{e\in E(C)} Z_e\bigg| > \beta \bigg).
\end{align}
Let $B$ be an integer with $B\geq 3$ and consider a \emph{$B$-gon} inscribed in the unit circle in the complex plane.  The distance of each side to the origin is $\cos\left( \frac{\pi}{B} \right).$ 
Thus by symmetry and a union bound,
\begin{align}
    p_{\fa}
    & \,\leq\, 
    \binom n k  \cdot  B \cdot \Q\bigg( \Re \bigg( 
    \sum_{e\in E(C)} Z_e  \bigg) > \beta\cos\Big( \frac{\pi}{B} \Big)  \bigg) \nonumber \\
    & \,=\, 
    \binom n k \cdot 
    B \cdot \Q\left( S_c > \beta\cos\Big( \frac{\pi}{B}   \Big) \right) \, ,  \label{eq:B-gon_bnd}
\end{align}
where $S_c \triangleq \sum_{e\in E(C)} \cos(X_e).$ For any $s > \E_{\Q}[S_c] = 0$, with $K=\binom k2,$
\begin{align} \label{eq: pfa-3}
    \Q (S_c \geq s) \,
    \, \stackrel{\mathrm{(A)}}{\leq} \, \inf_{\theta > 0} \, \exp\left\{-\theta \, s + K \log \mathbb{E}\big[e^{\theta \cos(X_{12})} \big] \right\}  \,
    \, \stackrel{\mathrm{(B)}}{\leq} \, \inf_{\theta > 0} e^{-\theta \, s + K \, \frac{\theta^2}{4} } \,
    \, \stackrel{\mathrm{(C)}}{=}    \,  e^{-s^2/K} \, ,
\end{align}
where (A) uses independence and identical distributions of $X_e$ under $\Q$, (B) uses that
\[ 
    \mathbb{E}_{\Q}[e^{\theta \cos(X_i)}] \,=\, \frac{1}{2\pi} \int_0^{2\pi} e^{\theta \cos t} \, \mathrm{d} t 
    \,=\, I_0(\theta)
    \,\leq\, e^{\frac 14 \theta^2},
\]
and (C) uses that $\theta^* = 2s/K$ optimizes the expression. Combining ~\eqref{eq:B-gon_bnd} and~\eqref{eq: pfa-3} yields:
\begin{align}  \label{eq:pfa_bnd_vM_coherence_community}
    p_{\fa} \leq \binom n k \cdot B \cdot \exp\left\{-\frac{\beta^2\cos^2(\pi/B)}{K} \right\} 
    \,\leq\, 
    B \cdot \exp \left\{k\log\left( \frac{ne}{k}\right)-\frac{\beta^2\cos^2(\pi/B)}K \right\} \, .
\end{align} 
Setting $\beta = (1-\frac{\epsilon}{4}) K \, A(\kappa)$ for an arbitrary $\epsilon$ with $0 < \epsilon < 1$ yields~\eqref{eq: pfa-vM-community-2}. 

\medskip

We turn to bounding the miss probability. For any set $E$ of edges,
\[ 
    \left|\sum_{e \in E} Z_e \right| = \max_{\theta} \sum_{e\in E}\cos(X_e - \theta) \geq \sum_{e \in E} \cos(X_e - \Theta^*),
\]
where $\Theta^*$ is the planted angle. Therefore, 
\begin{align} \label{eq: pmiss-vM-community-1}
    p_\miss = \P\bigg(\max_{C: |C| = k}\bigg|\sum_{e \in E(C)} \!\! Z_e  \bigg| \leq \beta \bigg) \leq \P\bigg( \,\bigg| \sum_{e\in E(C^*)} \!\! Z_e \bigg| \leq \beta\bigg) \leq \P \bigg( \sum_{e\in E(C^*)} \!\!\! \cos(X_e - \Theta^*) \leq \beta\bigg) \, ,
\end{align}
where $C^*$ is the planted set. It is easily verified that the random variables $\{ \cos(X_e - \Theta^*)\}_{e \in E(C^*) }$ are i.i.d. and supported on $[-1,1]$, with mean
\[ 
\int_0^{2\pi} \cos(x - \Theta^*) \cdot \frac{\exp(\kappa \, \cos(x - \Theta^*))}{2\pi I_0(\kappa)} \, \d x \stepa{=} \frac{1}{I_0(\kappa)} \cdot \left[\frac{1}{2\pi} \int_0^{2\pi} \cos(y) \, e^{\kappa \cos(y) } \, \d y \right] \stepb{=} \frac{I_1(\kappa)}{I_0(\kappa)} =: A(\kappa),
\]
where (a) uses the change of variable $y = x - \Theta^*$ and (b) uses the definition of the modified Bessel function $I_1(\cdot)$. Applying Hoeffding's inequality yields for any $u>0$,
\begin{align} \label{eq: pmiss-vM-community-2}
    \P\, \bigg(\sum_{e\in E(\mathcal{C}^*)}\cos(X_e-\Theta^*) \leq   K \, A(\kappa)-u\bigg) \, \leq \, \exp\!\bigg(-\frac{u^2}{2K}\bigg).
\end{align}
In particular, combining~\eqref{eq: pmiss-vM-community-1} and~\eqref{eq: pmiss-vM-community-2} and setting $u = \epsilon K\, A(\kappa)/4$ and $\beta = (1-\frac{\epsilon} 4) K \, A(\kappa)$ yields~\eqref{eq: pmiss-vM-community-3}.

\subsection{Converse for the von Mises community model}

Let $K\equiv\binom{k}{2}$. As in the flat model, the conditional likelihood ratio for a fixed candidate $(C,\theta)$ is given by 
\[
    L_{C,\theta}(X)
    = I_0(\kappa)^{-K}\exp\Big\{\kappa\sum_{e \in E(C)}\cos(X_e-\theta)\Big\} \, ,
\]
where $E(C)$ is the set of $\binom{k}{2}$ unordered pairs inside $C$. The (unconditional) likelihood
ratio is the mixture $L(X)=\mathbb{E}_{C,\Theta}[L_{C,\Theta}(X)]$. 

Let $(C,\Theta)$ and $(C',\Theta')$ be independent draws of the planted parameters. Proceed as in the flat model, 
using independence of edges under $\Q$:
\[
    \mathbb{E}_\Q\big[L_{C,\Theta}(X) \, L_{C',\Theta'}(X)~\big| ~C,\Theta, C', \Theta' \big]
    =\bigg[\frac{I_0\big(2\kappa\cos\frac{\Theta-\Theta'}{2}\big)}{I_0^2(\kappa)}\bigg]^{\binom S2} \, ,
\]
where $S \stackrel{\triangle}{=} |C\cap C'|$ and we use the fact that $|E(C)\cap E(C')|=\binom S2$ is the number of overlapping intra community edges. 
Averaging over $(C,\Theta)$ and $(C',\Theta')$ thus yields
\begin{align} \label{eq:matrix-second-moment-exact}
    \mathbb{E}_Q[L^2]
    \,= \, \mathbb{E}_{C,C'}\bigg[ \frac{1}{2\pi}\int_0^{2\pi} \, \Bigg(\frac{I_0\big(2\kappa\cos\frac{\phi}{2}\big)}{I_0^2(\kappa)}\Bigg)^{\binom{S}{2}}\, \d\phi\bigg], 
    ~~~~
    \phi\equiv \Theta-\Theta' \, .
\end{align}
We further upper bound by maximizing over $\phi$. Since $I_0\left(2\kappa\cos\frac {\phi} 2\right) \leq I_0(2\kappa)$ for all $\phi$,
\[
    \mathbb{E}_\Q[L^2]  \, \leq \, \mathbb{E}_{C,C'} \left[\,R(\kappa)^{\binom{S}{2}}\,\right] \, ,
\]
where $R(\kappa) \triangleq \frac{I_0(2\kappa)}{I_0^2(\kappa)}. $ 
The random variable $S$ has the $\mathrm{hypergeometric}(n,k,k)$ distribution. By convex ordering, the $z$--transform of this distribution is dominated by the $z$-transform of the $\mathrm{Binom}(k,k/n)$ distribution:
\[
    \mathbb{E}\big[z^{S}\big]\ \le\ \big(1-p+p z\big)^{k} \, , ~~~~~~ \text{for all } z\geq 0.
\]
Using $\binom{s}{2}\leq \frac{k-1}{2}\,s$ for $0\leq s\leq k$,
\begin{align}
    R(\kappa)^{\binom{S}{2}} 
    \,=\, 
    \exp \left( \big(\log R(\kappa) \big) \cdot \binom{S}{2}\right)
    \leq \exp \Big(\tfrac{(k-1)\log R(\kappa)} 2 S\Big) \, . \label{eq: log-ineq-loose}
\end{align}
Setting $t\equiv \tfrac{(k-1)\log R(\kappa)}{2}\,$ and $z=e^{t}$, we obtain
\[
    \E_\Q[L^2]\ \leq \ \E\big[e^{t S}\big]
    \, \leq \,   \big(1-p+p e^{t}\big)^{k}
    \, \leq \,   \exp\Big(k\,p\,(e^{t}-1)\Big)
    \, = \,      \exp\Big(\frac{k^2}{n}\big(e^{t}-1\big)\Big) \, .
\]
This proves the bound 
\begin{align}
\label{eq:matrix-second-moment-bound}
    \E_\Q[L^2] \,\leq\,  \exp \Big\{\frac{k^2}{n}\Big(\exp\!\big(\tfrac{(k-1)\log R(\kappa)}{2}\big)-1\Big)\Big\}\, .
\end{align}

Therefore, $\E_\Q[L^2]\to 1$ (and so detection is impossible) whenever the exponent in~\eqref{eq:matrix-second-moment-bound} goes to $0$. This is precisely the necessary condition~\eqref{eq:matrix-necessary-general}. This completes the proof of Theorem~\ref{thm:vonmises-community}(b).


\appendix
\section{Bounds and approximations of Bessel functions} \label{app:function_bounds}

This appendix collects bounds and asymptotic values related  to the modified Bessel functions of the first kind, $I_0(\kappa).$
\begin{itemize}

\item  
The power series expansion of $I_0$ and comparison with the power series expansion of $e^x$ evaluated at $x= {\kappa^2}/{4}$ gives, for all $\kappa,$
\begin{align*}
I_0(\kappa) & = \sum_{k=0}^\infty \frac 1 {(k!)^2}\left( \frac{\kappa}{2}  \right)^{2k}  \leq e^{\frac 14 \kappa^2}   
\end{align*}

\item  
Let $\kappa > 1/2.$  For $\Theta$ uniformly distributed over $[0,2\pi]$, the probability density of $\cos \Theta$ is $\frac 1 {\pi\sqrt{1-t^2}}$ for $-1 < t < 1.$ This gives the integral relation (known classically \cite{NIST:DLMF}) and bound:
\begin{align*}
I_0(\kappa) & = \frac 1 {\pi} \int_{-1}^1 \frac{e^{\kappa t}}{\sqrt{1-t^2} } \, \d t  \geq  \frac 1 {\pi} \int_{1 - \frac 1 {\kappa}}^1 \frac{e^{\kappa t}}{\sqrt{1-t^2} } \, \d t 
\end{align*}
Using $1-t^2 \leq 2(1-t)$ and $e^{\kappa t} \geq e^{\kappa}(1-(1-t))\kappa$ over the interval of integration and integrating yields
\begin{align*}
    I_0(\kappa) \geq \frac{2\sqrt{2}}{3\pi}\frac{e^{\kappa}}{\sqrt{\kappa}}
    \geq (0.3) \frac{e^{\kappa}}{\sqrt{\kappa}}.
\end{align*}
This lower bound is close to the large $\kappa$ asymptotics:
\begin{align*}
    I_0(\kappa) \sim \frac{e^{\kappa}}{\sqrt{2\pi \kappa}}
    \approx (0.39894)  \frac{e^{\kappa}}{\sqrt{2\pi \kappa}} \, .
\end{align*}

Accounting for the remainder term, it holds that as $\kappa \to \infty$,
\begin{align} \label{eq: I0-remainder}
    I_0(\kappa) = \frac{e^\kappa}{\sqrt{2\pi \kappa}}\big( 1+ O(\kappa^{-1})\big) \, .
\end{align}

\item  
Let
\begin{align*}
A(\kappa) = 
    \frac{1}{I_0(\kappa)} \cdot \left[\frac{1}{2\pi} \int_0^{2\pi} \cos(y) \, e^{\kappa \cos(y) } \, \d y \right] = \frac{I_1(\kappa)}{I_0(\kappa)}
\end{align*}
A Taylor expansion yields (see also Figure~\ref{fig: A-and-R-v2})
\begin{itemize}
    \item As $\kappa\to 0$, $A(\kappa) = \frac \kappa 2 + O(\kappa^3)$.
    \item As $\kappa\to \infty$, $A(\kappa) \to 1$.
\end{itemize}

\begin{figure}[t]
    \centering
    \includegraphics[width=0.6\linewidth]{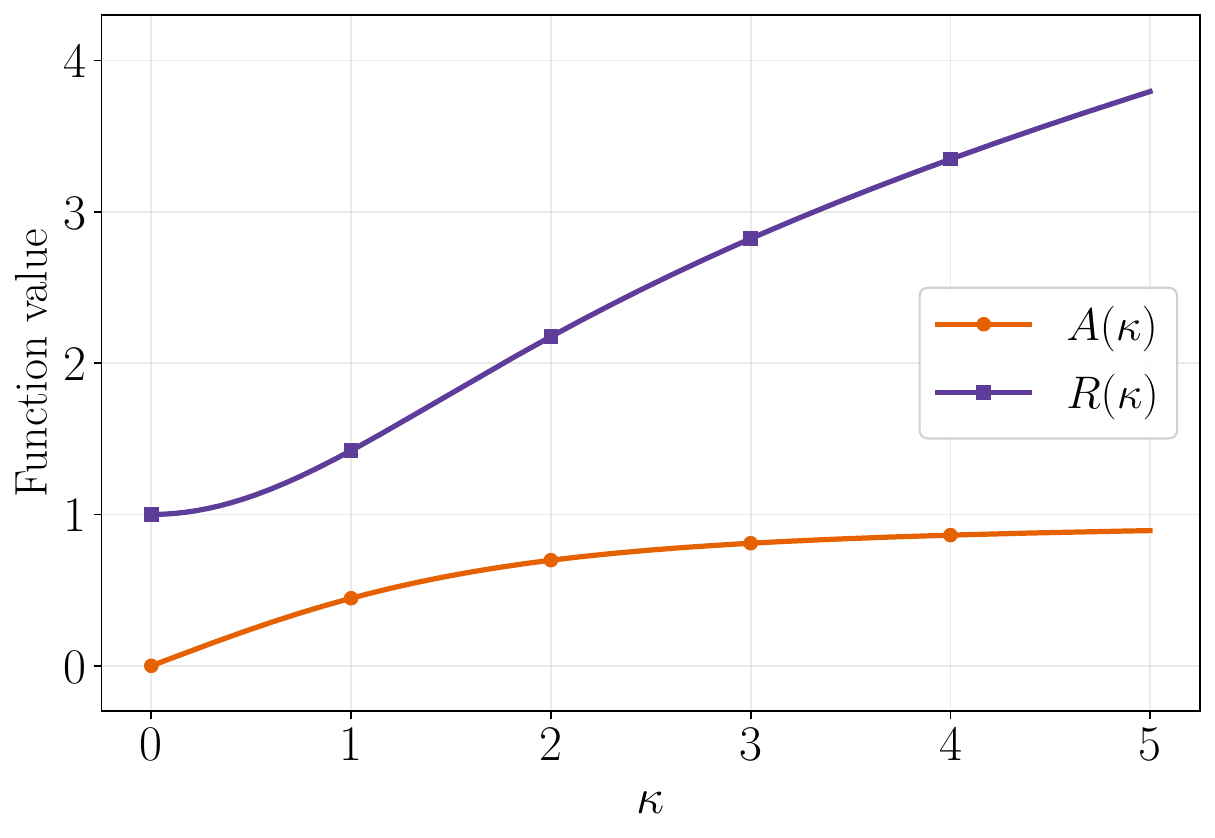}
    \caption{The functions $A(\kappa)$ and $R(\kappa)$}
    \label{fig: A-and-R-v2}
\end{figure}

\item 
Let $\rho_\kappa(\theta) \triangleq {I_0(2\kappa \cos\theta)}/{I_0^2(\kappa)}.$  Lemma \ref{lem:vm_rho_mean_one}    shows that  for every $\kappa > 0$,
$
\frac 1 {2\pi} \int_0^{2\pi} \rho_\kappa(\theta) \,\d\theta = 1.  
$

\item 
Let  $R(\kappa) \triangleq {I_0(2\kappa)}/{I^2_0(\kappa)}.$
The maximum of $\rho_\kappa(\theta)$ over $\theta$ is $R(\kappa) = \rho_\kappa(0).$ The asymptotic behavior of $I_0$ described above implies the following asymptotic behavior of $R.$
\begin{itemize}
    \item As $\kappa\to 0,$ $R(\kappa) = 1 + \frac{\kappa^2} 2 + O(\kappa^4).$
    \item As  $\kappa\to \infty,$ $R(\kappa)  = \sqrt{\pi\kappa} \, \big( 1 + O(\kappa^{-1}) \big) $.
\end{itemize}

\item 
As $\kappa\to 0$,
\begin{align*}
\KL( \mathrm{von Mises}(0,\kappa) || \uniform ([0,2\pi]) 
&= \kappa \frac{I_1(\kappa)}{I_0(\kappa)} - \log I_0(\kappa) \\
&= \frac{\kappa (\kappa/2 + o(\kappa))}{I_0(\kappa)} - \log I_0(\kappa) \\
& = \frac{\kappa^2} 2 - \frac{\kappa^2} 4 + o(\kappa^2) = \frac{\kappa^2} 4 + o(\kappa^2)
\end{align*}

\end{itemize}

\section{Generalized likelihood ratio tests} \label{app: GLRT}

\subsection{The interval test for the hard-cluster flat model}  \label{sec:scan_as_GLRT}
For $\theta\in[0,1]$, define the count
\[
    N(\theta) \triangleq \sum_{i=1}^N \ones{X_i\in [\theta,\theta+ 2\pi\tau] ~~ \mathrm{(mod~1) }},
    ~~~~~~
    M_\tau(X) \triangleq \sup_{\theta\in[0,2\pi]} N(\theta),
\]
i.e. $M_\tau(X)$ is the largest number of points that can be captured in an interval of length $2\pi\tau$. Recall that the interval test above with window $2\pi\tau$ and threshold $\gamma$ is
\[
    \phi_{\tau,\gamma}(X) \triangleq \ones{\,M_\tau(X)\geq \gamma\,},\qquad \gamma\in\{1,\dots,n\}.
\]

\paragraph{Likelihood ratios.}
Recall the conditional likelihood ratio given $(S,\theta)$,
\begin{align}
    L_{S,\theta}(x) 
    =\prod_{i\in S}\frac{\ones{x_i\in I_\theta}}{\tau}
    =\tau^{-K}\,\ones{\,x_i\in I_\theta\ \forall i\in S\,}.
\end{align}
Averaging over $S$ and $\Theta$ gives the (unconditional) likelihood ratio,
\begin{align} \label{eq: LR}
\begin{aligned}
    L(x)& \triangleq \frac{\d P}{\d Q}(x)
    =\frac{1}{\binom{N}{K}\,\tau^K}\int_0^{2\pi}
    \sum_{\substack{S\subset[N]\\ |S|=K}}\mathbf{1}\big\{\,S\subset\{i:x_i\in I_\theta\} \big\}\, \frac{\d\theta}{2\pi}
    \\
    &=\frac{1}{\binom{n}{K}\,\tau^K}\int_0^{2\pi} \binom{N(\theta)}{K}\, \frac{\d\theta}{2\pi}.
\end{aligned}
\end{align}
The last step used that the number of $K$-subsets contained in a set of size $N(\theta)$ is $\binom{N(\theta)}{K}$.  We consider two GLRTs, depending on how we eliminate the parameters $(S,\theta)$, as follows:

\medskip\noindent
\begin{itemize}
    \item[(A)] \textit{Maximize over both $S$ and $\theta$.} In this case,
    \begin{align}
        T_{A}(x) \triangleq \sup_{\theta\in[0,1]}\ \sup_{\substack{S\subset[N]\\|S|=K}} L_{S,\theta}(x)
        \stepa{=}
        \sup_{\theta}\ \tau^{-K}\,\ones{\,N(\theta)\ge K\,}
        \stepb{=}
        \tau^{-K}\,\ones{\,M_\tau(x)\ge K\,},
    \end{align}
    where (a) used that given $\theta$, there exists some index set $S$ of points in $[\theta,\theta+2\pi\tau]$ iff $N(\theta)\geq K$, and (b) uses the definition of $M_\tau(x)$.
    
    Therefore, the decision rule in this case is precisely to compare $T_{A}$ to any $t\in(0,\tau^{-K}]$:
    \[
        \text{reject }H_0\ \Longleftrightarrow\ M_\tau(X)\ge K,
    \]
    i.e. the interval test with threshold $\gamma=K$.
    \item[(B)] \textit{Average over $S$, then maximize over $\theta$.} 
    Following~\eqref{eq: LR}, we have
    \begin{align}
        T_{B}(x) \triangleq \sup_{\theta\in[0,1]} L_\theta(x)
        =
        \frac{1}{\binom{N}{K}\,\tau^{K}}\ \sup_{\theta}\binom{N(\theta)}{K}.
    \end{align}
    Since $m \mapsto \binom{m}{K}$ is strictly increasing for $m\ge K$, thresholding $T_{B}$ is equivalent to thresholding $M_\tau(X)$:
    \begin{align}
    T_{B}(x)>t\ \Longleftrightarrow\ M_\tau(X)\ge \gamma
    \quad\text{for some integer }\gamma\ge K \text{ depending on }t,
    \end{align}
    i.e. the interval test with tunable $\gamma\ge K$.
\end{itemize}

\subsection{The coherence test for the von Mises community model}
\label{sec:coherence_test_vM_community}

Consider the detection problem for the von Mises community model with parameters $n,k, \kappa.$
The log likelihood ratio given $\Theta^*=\theta$ and $C^*=C$ is given by:
\begin{align*}
    \log L_{\theta,C}(X) & = \kappa \sum_{e\in E(C)}  \cos(X_e - \theta)  - K\log(I_0(\kappa))  \\
    & =  \kappa \cos(\theta) \sum_{e\in E(C)}  \cos(X_e) + \kappa \sin(\theta) \sum_{e\in E(C)}  \sin(X_e) - K\log(I_0(\kappa))
\end{align*}
For each edge $e$, let $Z_e = e^{\i X_e} = \cos(X_e) + \i \sin(X_e)$
Maximizing over $\theta$ yields:
\begin{align*}
    \max_{\theta \in [0,2\pi] } \log L_{\theta,S}(X) & = \kappa \bigg|\sum_{e \in E(C)} Z_e \bigg|  - K\log(I_0(\kappa)) 
\end{align*}
If we further maximize over $C$ and select a threshold $\beta$ we arrive at the GLRT test
\[
    \psi_n(\beta) \triangleq \mathbf{1} \left\{ \,\, \max_{\substack{C\subseteq [n] \\ |C| =k } } \,\,\left| \sum_{e\in E(C)} Z_e\right| \,\geq \, \beta \right\}.
\]
This is our {\em coherence test} for the von Mises community model.

\section{On polynomial-time tests} \label{app: poly-time}

The \emph{Rayleigh test}, defined below, has been successfully used in literature for detecting that a sample of observations is drawn from a von Mises distribution versus a uniform distribution. We show that it is not as effective in the planted setting, for the von Mises community model.

Let $Z_{ij} = e^{\i X_{ij}}$, and consider the test statistic
\[
    T_{\text{Rayleigh}} \,\triangleq\, \bigg|\sum_{1\le i < j \le n} Z_{ij}\bigg| \, .
\]
We analyze for threshold $\beta$ the test $\mathbf{1}\{T_{\text{Rayleigh}} \geq \beta\}$. Let $N_E = \binom{n}{2}$ be the total number of edges.

\textit{Bounding false alarm probability.~} Under $H_0$, the $Z_{ij}$ are i.i.d. uniform on the unit circle. Using the concentration bound~\eqref{eq:B-gon_bnd} and~\eqref{eq: pfa-3}:
\begin{equation}
p_{\fa} = \mathbb{Q}\,(T_{\text{Rayleigh}} \geq \beta) \leq 4 \exp\left(-\frac{\beta^2}{2N_E}\right).
\end{equation}

\textit{Bounding the miss probability.~} Under $H_1$, let $\Theta$ be the planted angle and let $m = \binom{k}{2}$. Using the same idea as~\eqref{eq: pmiss-vM-community-1}~and~\eqref{eq: pmiss-vM-community-3}, we have that the expected value of the sum $S = \sum_{i<j} Z_{ij}$ given $\Theta$ is $\mathbb{E}_P[S|\Theta] = m A(\kappa) e^{\i\Theta}$. Let $\mu_1 = m \,  A(\kappa)$.
We lower bound the statistic by projecting onto the true direction $\Theta$:
\[
    T_{\text{Rayleigh}} = |S| \, \geq \,  \Re (S e^{-\i\Theta}) \, =\,  \sum_{i<j} \cos(X_{ij}-\Theta)  \, \triangleq\, S_\Theta \, .
\]
$S_\Theta$ is a sum of $N_E$ independent real random variables bounded in $[-1, 1]$, with mean $\mu_1$. By Hoeffding's inequality:
\[
    p_{\miss} \leq P(S_\Theta < \beta) \leq \exp\left(-\frac{(\mu_1-\beta)^2}{2N_E}\right) \, ,
\]
provided $\beta < \mu_1$.

To ensure both error probabilities vanish, we must be able to select $\beta$ such that $\beta^2 = \omega\big(N_E\big)$ (for $p_{\fa}$) and $\beta < \mu_1$ (for $p_{\miss}$). This requires $\mu_1^2 = \omega \big( N_E \big)$. Choose $\beta = \mu_1/2$, so the total error probability
\[
    p_{\fa} + p_{\miss} \le 5 \exp\left(-\frac{\mu_1^2}{8N_E}\right) = 5 \exp\left(-\frac{m^2 A^2(\kappa)}{8N_E}\right).
\]
Substituting $m\approx k^2/2$ and $N_E\approx n^2/2$, the sufficient condition for this test to succeed is
\[
    k^4 A^2(\kappa) = \omega \big( n^2 \big) \quad \mbox{or equivalently} \quad k^2 A(\kappa) =  \omega \big( n \big) \, .
\]
This test is computationally efficient but requires a larger community size (e.g., $k  = \omega \big(\sqrt{n} \big)$ if $A(\kappa)=\Theta(1)$), compared to the (computationally hard) coherence test.

\section{On the relationship between interval and coherence test} \label{app-relationship-test}

We comment on the relationship between the interval test and the coherence test for the community models. In Theorem~\ref{thm:vonmises-community}, sufficient conditions for strong detection for the von Mises community model are given for $k=o(\log n)$ using the interval test, and for $k=c\log n$ and $k=\omega(\log n)$ using the coherence test.  While those sections are complementary, the estimators themselves are rather similar in the regime of small $k$ (and so large $\kappa$) as we explain in this section.

The statistic used in the coherence test for a given $C\subset [n]$ with $|C|=k$ satisfies
\begin{align*}
    \bigg|\sum_{e\in E(C)}  Z_e  \bigg| = \sum_{e\in E(C)} \cos(X_e-\hat{\Theta}_C)
    ~\mbox{  where  }~ \hat{\Theta}_C = \arg\max_{\theta}  \sum_{e\in E(C)}  \cos(X_e-\theta).
\end{align*}
If the threshold $\beta$ has the form $K-\epsilon$ for some $\epsilon\in (0,1)$ and if the threshold is exceeded:
\begin{align}
  \sum_{e\in E(C)}  \cos(X_e-\hat{\Theta}_C)  \geq K - \epsilon     \label{eq:exceed_threshold}
\end{align}
then $\cos(X_e - \hat{\Theta}_C)  \geq 1 - \epsilon$ for each $e.$ If $\epsilon$ is close to zero then by the Taylor approximation of $\cos$ we have
\begin{align*}
    \bigg|\sum_{e\in E(C)}  Z_e  \bigg|
    & \approx K - \frac 1 2 \sum_{e\in E(C)} |X_e -  \hat{\Theta}_C |^2  \\
    \hat{\Theta}_C &\approx \arg\min_{\theta} \sum_{e\in E(C)} |X_e - \theta|^2
\end{align*}
Therefore, if we define $V_C = \min_{\theta} \frac 1 {K-1} \sum_{e\in E(C)} |X_e-\theta|^2,$ the event \eqref{eq:exceed_threshold} is approximately the same as $V_C \leq \frac{2\epsilon}{K-1}.$  If the values $X_e$ for $e \in E(C)$ are all on the same half of the unit circle then $V_C$ is the usual unbiased estimator of variance for real valued observations. With this motivation we define the {\em variance test} to be
\begin{align*}
    \psi_n(\tau) \triangleq \mathbf{1} \left\{  \min_{\substack{C\subseteq [n] \\ |C| =k } } V_C(X)  \leq \sigma^2  \right\} \, .
\end{align*}
for some threshold $\sigma^2$.

In summary, the coherence test with threshold $\beta = K-\epsilon$ is approximately the same as the variance test with threshold $\sigma^2 = \frac{2\epsilon}{K-1}$.   In turn, the variance test is rather similar to the interval test with threshold $K$ and interval width $\tau = \sigma$ in the sense that the square root of sample variance and the width of the range of $(X_e : e\in E(C))$ are two similar measures of the spread of those random variables.

\section{On knowing~\texorpdfstring{$\Theta^*$}{}} \label{app: knowing-theta}

For the community models, our achievability results do not rely on knowledge of the planted phase $\Theta^*$, while  our converse bounds continue to hold even if $\Theta^*$ is revealed to the decision maker. For the flat models, we can characterize how access to $\Theta^*$ can improve performance by reducing the optimal error probabilities, as quantified in the following proposition.

\medskip

\begin{proposition} \label{prop:unif_flat_detection_known_theta}
Consider the detection problem for the flat uniform model such that
$\Theta^*$ is known to the decision maker.   If $K^2\leq N$ and $\tau = o\left({K^2} /N \right)$ then strong detection is possible.  If $K^2\leq N$ and  $\tau = \omega \left({K^2} / N \right)$ then weak detection is impossible.
\end{proposition}

\medskip

\begin{remark}
Comparing Corollary~\ref{cor:flat_uniform_achievability} and Proposition \ref{prop:unif_flat_detection_known_theta} shows that the value of knowing $\Theta$ is that $\tau$ can be larger by a factor $N^{\frac 1 {K-1}}$ for constant $K$ and by a factor $\log N$ for $K=N^{\alpha}$ for $0 < \alpha < 1/2.$
\end{remark}

\medskip

\begin{proof}
Without loss of generality, by symmetry, assume it is known that $\Theta^* = 0.$

(Achievability)
Suppose $K^2\leq N$ and $\tau = o\left({K^2}/ N \right).$
Let $Y$ denote the number of observations that fall into the interval $[0,\tau]$ and consider the test with threshold $\gamma$ based on $Y.$   The error probabilities are given by:
\begin{align*}
    p_{\miss} = P(K + \mbox{binom}(N-K,\tau) < \gamma) ~~\mbox{and}~~ p_{\fa} = P(\mbox{binom}(N,\tau) \geq \gamma).
\end{align*}
Note $p_{\miss}$ here is equal to the upper bound on $p_{\miss}$ used  in Section \ref{sec-2-achievability}.  Hence, we can apply the same upper bound on $p_{\miss}$ as in that section.  Also, $p_{\fa}$ is less than $\Q(B_1)$ which is less than the right hand side of \eqref{eq:pfa_bnd} (for $\gamma \geq 1 + (N-1)\tau$) and the right hand side of \eqref{eq:pfa_unif_up_bound1}, both with the leading factor of $N$ removed. 
If $K$ is constant then $n\tau \to 0$ and we take $\gamma =K$ leading to $p_{\miss}=0$ and $p_{\fa}\to 0$ by \eqref{eq:pfa_unif_up_bound1} with the factor of $N$ removed.
If $K\to\infty$ such that $K^2\leq N$ and $\tau = o(K^2/N)$ then~\eqref{eq:suff_cond_1} and~\eqref{eq:suff_cond_2} with the term $-\log(N)$ removed both hold and together are sufficient to prove $p_{\miss}\to 0$ and $p_{\fa}\to 0.$   

(Converse) Suppose that  $K^2\leq N$ and  $\tau = \omega \left(\frac{K^2} N \right)$. 
We follow the moment method of the proof of Theorem~\ref{thm:hard-cluster-flat}(b).   With it known that $\Theta^*\equiv 0$, only the case $\theta=\theta'=0$ is relevant so the variable $\delta$ is equal to $\tau.$  We thus find:
\begin{align*}
    E_{\Q}[L^2] & =   \sum_{j=0}^K P\{ |S\cap S'| = j\} \,  \tau^{-j}  \leq  \sum_{j=0}^K P\left\{ \mbox{binom}\left( K, \frac K N\right) = j\right\} \, \tau^{-j}  =  \left( 1 - \frac K N + \frac{K}{N\tau}\right)^K  \\
    & \leq \exp\left( \frac{K^2}{N\tau} \right)  \to 1.
\end{align*}
Therefore, weak detection is impossible.
\end{proof}

\medskip

\begin{remark}
    The hypothesis test used in the above proof of Proposition \ref{prop:unif_flat_detection_known_theta} is the maximum likelihood ratio test and is hence Pareto optimal.   So another method of proof of the converse part of the Proposition would be to show $\min_{\gamma} p_{\fa} + p_{\miss} \to 1/2$ for that test.
\end{remark}

\section{Community recovery and connection to stochastic block models} \label{app: SBM}

The community models in this paper differ from stochastic block models for a planted dense community in that the marginal distributions  of the edge labels is the same under both distributions, $\P$ and $\Q.$
However, if a genie were to reveal the value of the parameter $\Theta$ then the community models in this paper fall into the realm of stochastic block models with general edge label distributions $P$ and $Q$ as in    \cite{wu2021planted,hajek2017information}.  Under that model for given $n,k,P,Q$ it is assumed there is a subset $C^*$ of cardinality $k$ drawn uniformly at random from among subsets of $[n]$ of cardinality $k.$   A symmetric $n\times n$ matrix $A$ with zero diagonal is observed such that for all $1\leq i < j \leq n$, $A_{ij}$ are independent and $A_{ij}\sim P$ if $i,j \in C^*$ and $A_{ij}\sim Q$ otherwise.   Necessary and sufficient information theoretic conditions for weak and strong detection are given in \cite{wu2021planted} and for weak and exact recovery are given in \cite{hajek2017information}.  The necessary conditions there translate to necessary conditions for the model here because the genie giving $\Theta$ can be ignored.   For simplicity we assume without loss of generality that $\Theta = 0.$

To apply \cite{hajek2017information} we let $P$ denote the von Mises distribution with parameters $\kappa, 0$ and Q be the uniform distribution on $[0,2\pi].$  
The likelihood ratio is given by $\frac{\d P}{\d Q}(x) = \frac{e^{\kappa\cos(x)}}{I_0(\kappa)}$ which is bounded if $\kappa$ is bounded in which case Assumption 1 of \cite{hajek2017information} holds by Lemma 1 in  \cite{hajek2017information}.   Thus, by Theorem 1 in \cite{hajek2017information} in a regime with $\kappa$ bounded a necessary condition on $\kappa$ for weak recovery is $K\cdot \KL(P||Q)\to \infty$ and $\lim\inf \frac{(K-1)\, \KL(P||Q)}{\log (n/ K)}\geq 2.$
Note that
\begin{align*}
    \KL(P||Q) = \E_P\left[\log \frac{\d P}{\d Q}\right] = \E_P[\kappa \cos(\Theta)] - \log I_0(\kappa) = \kappa \,A(\kappa) - \log I_0(\kappa)
\end{align*}
For $\kappa \to 0$, $\KL(P||Q) \sim \frac{\kappa^2}{4}.$   Thus, if we consider $k$ varying with $n$ such that  $\Omega(\log n) \leq k \leq o(n)$ we find that a necessary condition for weak recovery is
$\lim\inf \frac{(k-1)\kappa^2}{\log (n/k) } \geq 2$, which is satisfied for example if $\kappa^2 \geq \frac{2\log(n/k)}{k}.$
[ We can also try using Theorem 2 of \cite{hajek2017information} to get a necessary condition for exact recovery.]

If $\Theta$ were known we could also apply results from \cite{HajekWuXu_MP_submat18} (related to \cite{deshpande2015finding}) to yield a polynomial time message passing algorithm for weak recovery and, after a cleanup procedure, exact recovery, if $\Omega(\sqrt{n}) \leq k \leq o(n)$ under a signal-to-noise ratio condition (that is not information theoretically tight for weak recovery but might be for exact recovery).   The analysis in those papers is based on the moments of the observations so we change variables and work with the observations $W_{ij}$ given by  $W_{ij}=\sqrt{2}\cos(X_{ij} - \Theta).$  The new observations form a sufficient statistic (assuming $\Theta$ is known) so there is no information theoretic loss.   We have that if $\{i,j\}\not\subset C^*$
then $\E[W_{ij}]=0$ and $\E[W^2_{i,j}]=1$ and if $\{i,j\}\subset C^*$ then $\mu = \E[W_{ij}] = \sqrt{2} A(\kappa) = {\kappa}/{\sqrt{2}} + O(\kappa^3)$ so the message passing algorithm succeeds if $\lambda = \frac{\kappa^2K^2}{2n} > \frac 1 e.$

The message passing algorithm can be adapted to work without knowing $\Theta$ as follows.   Given a large integer $B$ we could run $B$ copies of the message passing algorithm in parallel, where the $b^{th}$ version is applied to the observation matrix $W^{(b)}=\sqrt{2}\cos(X_{ij} + \frac{b}{2\pi}).$   Then for each $b$,  $\mu_b = E[W_{ij}|\Theta^*=\theta] = \cos(\theta - \frac{b}{2\pi})A(\kappa)$  Therefore, $\max_b \mu_b \geq \cos\left(\frac{1}{4\pi B}\right) A(\kappa). $

Thus, if $\lambda$ is fixed with $\lambda > \frac 1 e$ then we select a positive integer $B$ large enough that $\lambda \cos\left(\frac{1}{4\pi B}\right)  > \frac 1 e.  $   Then if $\kappa$ is chosen so that
$\lambda = \frac{\kappa^2 K^2}{2n}$ the combined message passing algorithm can achieve both error probabilities converging to zero -- that is, strong detection in polynomial time.


\bibliographystyle{alpha}
\bibliography{graphical_combined.bib}

\end{document}